\documentclass[11pt, %psamsfonts
]{amsart}

\usepackage{color, hyperref}

\usepackage[a4paper,margin=28mm]{geometry}
\usepackage{amsfonts,amssymb,amscd}
\usepackage{graphicx,mathrsfs} % para usar el comando \mathscr{}
\usepackage{subfigure}
\newtheorem{theorem}{Theorem}[section]
\newtheorem{lemma}[theorem]{Lemma}

\newtheorem{definition}[theorem]{Definition}%\renewcommand{\thedefinition}{}

\theoremstyle{remark}

\newtheorem{remark}[theorem]{\bf Remark}
\newtheorem{example}[theorem]{\bf Example}

\renewcommand{\leq}{\leqslant}
\renewcommand{\geq}{\geqslant}

\newcommand{\ptl}{\partial}

\newcommand{\rr}{\mathbb{R}}

\newcommand{\nn}{\mathbb{N}}


\numberwithin{equation}{section}

\graphicspath{{Figuras/}}

\begin{document}

\title[Optimal divisions of a convex body]
{Optimal divisions of a convex body}

\author[A. Ca\~nete]{Antonio Ca\~nete}
\address{Departamento de Matem\'atica Aplicada I and IMUS \\ Universidad de Sevilla}
\email{antonioc@us.es}

\author[I. Fern\'andez]{Isabel Fern\'andez}
%\address{Departamento de Matem\'atica Aplicada I \\ Instituto de Matemáticas IMUS \\ Universidad de Sevilla}
\email{isafer@us.es}

\author[A. M\'arquez]{Alberto M\'arquez}
%\address{Departamento de Matem\'atica Aplicada I \\ Instituto de Matemáticas IMUS \\ Universidad de Sevilla}
\email{almar@us.es}

\begin{abstract}
For a convex body $C$ in $\mathbb{R}^d$ and a division of $C$ into  convex subsets $C_1,\ldots,C_n$, we can consider
$max\{F(C_1),\ldots, F(C_n)\}$ (respectively, $min\{F(C_1),\ldots, F(C_n)\}$),
where $F$ represents one of these classical geometric magnitudes:  the diameter, the minimal width, or the inradius.
In this work we study the  divisions of $C$  minimizing (respectively, maximizing) the previous value, as well as other related questions.
\end{abstract}

%\begin{keyword}
%convex body\sep optimal division\sep diameter\sep minimal width\sep inradius
%\MSC[2010]{52A20, 52A40}
%\end{keyword}

\maketitle

\section{Introduction}

Finding \emph{the best} division of a given set, from a geometric point of view, is an interesting non-trivial question deeply studied in different settings,  specially in the last decades, which may yield striking results in some situations.
%For a given set $C$ in $\rr^d$,  determining the most appropriate division (with respect to a given  functional and under some natural restrictions) constitutes a non-trivial task and yields striking results in some situations.
%The general statement of the problem treated in this work is the following. Consider a geometric mag $F$ be one of the classical well-known geometric magnitudes (perimeter, area,  volume, diameter, inradius, circumradius, width), and let $C$ be an arbitrary set in $\rr^d$. For any division of $C$ into $n$ connected subsets $C_1,\ldots,C_n$, we can consider the largest value among $F(C_1),\ldots,F(C_n)$. This assigns a positive value to each division of $C$, which can be understood as a measure of the quality for the divisions. In this setting, our main interest is finding the divisions of $C$ providing the smallest previous value (if they exist). Observe that this leads us to a min-Max type problem for the magnitude $F$ and for the divisions of $C$ into $n$ subsets. Following the usual setting for this kind of questions, we will assume the set $C$ to be a \emph{convex bodies}  (compact convex set with non empty interior) in $\rr^n$.

In this line, \emph{Conway's fried potato problem} (\cite[Problem~C1]{cfg}) looks for the division of a convex body $C$ in $\rr^d$ into $n$ subsets (under the additional restriction of using $n-1$ successive hyperplane cuts) minimizing the largest inradius of the subsets. This problem was solved by A.~Bezdek and K.~Bezdek in 1995, proving that a minimizing division is given by means of $n-1$ parallel hyperplane cuts, equally spaced in the slab determining the minimal width of a certain rounded body associated to $C$. We note that this construction is implicit, and the \emph{optimal value} associated to this problem is determined  theoretically~\cite[Th.~1]{bb}.

A similar question for the diameter magnitude has been also considered in the planar setting in several joint works by one of the authors:
% quitar esa referencia a mi?
for the family of centrally-symmetric planar convex bodies and arbitrary divisions into two subsets, necessary and sufficient conditions for being a minimizing division can be found in~\cite{CS} (see also~\cite{MPS}). Moreover, for a $k$-rotationally symmetric planar convex body $C$, where $k\in\mathbb{N}$,  $k\geq3$, a minimizing $k$-partition (which is a particular type of division into $k$ subsets, by means of $k$ curves which meet at an interior point of $C$) is described in~\cite[Th.~4.5]{extending} for any $k\geq 3$, as well as a minimizing general division (without  restrictions) into $k$ subsets when $k\leq6$~\cite[Th.~4.6]{extending}.
Additionally, a related approach for general planar convex bodies has been treated in~\cite{CG}.
%, a sharp isodiametric inequality for divisions into two subsets is given in~\cite[Th.~1.1]{CG}, characterizing the unique optimal body, and determining its associated minimizing divisions~\cite[Re.~3.2]{CG}.

These two previous questions can be regarded as particular cases of the following {\em min-Max} and {\em Max-min} type problems:

{\bf Problem:} {\em Given a geometric magnitude $F$ and a convex body $C\subset\rr^d$, which are the divisions of $C$, if any, minimizing (resp., maximizing)  the largest (resp., the smallest) value of $F$ on the subsets of the division?}

The present work is devoted to study  this problem when $F$ is the diameter, the minimal width and the  inradius. % of a convex body.
Following the original statement of Conway's fried potato problem, we will consider divisions determined by  \emph{successive hyperplane cuts} (see~\cite[\S.~2]{bb} for a more precise description). 
% \textcolor{blue}{aÃ±adir la referncia a donde se explica lo de los cortes sucesivos}
We will address the question of existence %and uniqueness
of an \emph{optimal division}, as well as its {\em balancing behaviour} (in the sense that all the subsets in those divisions have the same value for the considered magnitude, see Section~\ref{sec:pre}). We will also give the optimal value of the magnitude $F$ when possible, or  upper and lower bounds when not. Moreover, for the family of convex polygons, we will provide an algorithm for computing the optimal value (and consequently, an  optimal division) for the min-Max problem for the inradius (Conway's fried potato problem), for which the solution was known only theoretically, as explained above. This algorithm is of quadratic order with respect to the number of sides of the polygon, see Subsection~\ref{subsec:mMi-algorithm}.

Our main results can be summarized as follows:

{\bf Theorem A (min-Max type problems):}  {\em Let $C$ be a convex body in $ \rr^d$, $F$ one of the following magnitudes: diameter ($D$), width ($w$), inradius ($I$),  and $n\geq 2$.
Then, there exists a division of $C$ into $n$ subsets (given by $n-1$ sucessive hyperplane cuts) minimizing the largest value of $F$ on the subsets. This optimal division can be chosen to be {\em balanced}. Moreover,
\begin{enumerate}
\item If $F=D$,  lower and upper bounds for the optimal value are given in~\eqref{eq:mMDbound2}.
\item If $F=w$, any optimal division is balanced and the optimal value is $w(C)/n$.%, see Theorem~\ref{prop:mMwbalanced}.%, see Theorem~\ref{prop:mMwov}.
\item If $F=I$, the optimal value is given in terms of the width of some {\em rounded body} associated to $C$ (\cite[Th.~1]{bb}), although an explicit sharp lower bound is given in~\eqref{eq:desi}.
\end{enumerate}

In the last two cases, optimal divisions for convex polygons can be found by means of algorithms of lineal and quadratic order with respect to the number of sides of the polygon,  respectively.}

%We point out that the obtained upper bound for $F=D$ is not sharp, and we propose a method for improving that bound, based on partitioning an auxiliary set containing the original one (see Theorem~\ref{prop:mMDbound3} and preceding paragraphs).  \\

{\bf Theorem B (Max-min type problems):} {\em Let $C$ be a convex body in $\rr^d$, $F$ one of the following magnitudes: diameter ($D$), width ($w$), inradius ($I$),  and $n\geq 2$.
Then, there exists a division of $C$ into $n$ subsets (given by $n-1$ successive hyperplane cuts) maximizing the smallest value of $F$ on the subsets (except possibly when $d=2$ and $F=D$%, see Theorem~\ref{prop:MmDexistencia2}
). This optimal division can be chosen to be \emph{balanced}.  Moreover,
\begin{enumerate}
\item If $F=D$,  any optimal division is balanced and the optimal value is $D(C)$.
\item If $F=w$, sharp lower and upper bounds for the optimal value are given in~\eqref{eq:Mmwbound1}.
\item If $F=I$, a lower bound for the optimal value is given in~\eqref{eq:Mmibound}.
\end{enumerate}
}
For the Max-min type problems, we also remark that  any optimal division when $F=w$ is balanced for $n=2$, and that the optimal value when $F=I$ can be expressed in an analogous way as in~\cite{bb} (that is, in terms of the optimal value for the Max-min problem for the width for a certain rounded body associated to $C$, see Theorem~\ref{prop:Mmiov}).

%In the case of the diameter, we stress that the optimal value can be calculated (it coincides with the diameter of the original convex body, see Theorem~\ref{prop:MmDov}), and existence of optimal divisions is assured in $\rr^d$ for $d\geq3$ (Theorem~\ref{prop:MmDexistencia1}), but not in the planar case (where we will need enough segments providing the diameter of the original convex body, see Theorem~\ref{prop:MmDexistencia2}). For the width magnitude, we provide sharp lower and upper bounds for the optimal value (Theorems~\ref{prop:Mmwbound1} and~\ref{prop:Mmwbound2}), as well as proving the existence of optimal divisions, see Theorem~\ref{prop:Mmwexistence}. And for the inradius, the existence for the Max-min problem is proved in Theorem~\ref{prop:Mmiexistence}, and the corresponding optimal value for divisions into two subsets is implicitly obtained by using the notion of rounded body, see Definition~\ref{def:rounded} and Theorem~\ref{prop:Mmiov}. For general divisions, an explicit  lower bound is established in Theorem~\ref{prop:Mmiovn}.

The paper is organized as follows.
Section~\ref{sec:pre} establishes the precise statement of our problems. % and shows some general results.
In Section~\ref{sec:mM} we consider the min-Max type problems, %for the diameter, the width and the inradius,
proving the results in Theorem A, whereas
Section~\ref{sec:Mm} is devoted to the corresponding Max-min type problems, proving  Theorem B. Finally, Section~\ref{sec:final} contains some related questions.%\textcolor{blue}{esto no sÃ© si lo Ã­bamos a quitar, creo que dijimos que no. Otra opcion seria comentar aqui los casos que se han quedado abiertos, que creo que despues de enunciar los teoremas A y B viene a huevo}

%%%%%%%%%%%%%%%%%%%%%%%%%%%%%%%%%%%%%%%%%%%%%%%%%%%%%%%%%%%%%%%%%%%%%%%%%%%%%%%%%%%%%%%%%%%%%%%%%%%%%%%%%%%%%%%%%%%%%%%%%%%%%%%%%%%%%%%

\section{General setting and preliminaries}
\label{sec:pre}

From now on, $C$ will denote a convex body (convex compact set with non-empty interior) in $\rr^d$, $d\geq 2$.

Following ~\cite{bb} (see also ~\cite[Subsec.~2.2]{RS}), an \emph{$n$-division} of $C$ will be a decomposition of $C$ into $n$ closed subsets $C_1,\ldots, C_n$, all of them with non-empty interior, %where the subsets $C_i$ are
given by $n-1$ \emph{successive hyperplane cuts}: along the division process, only one subset is divided into two by each hyperplane cut (see Figure~\ref{fig:successive}). In particular, all the subsets of an $n$-division are convex, and  the intersection between two adjacent subsets is always a piece of hyperplane.

\begin{figure}[ht]
    \centering
    \includegraphics[width=0.75\textwidth]{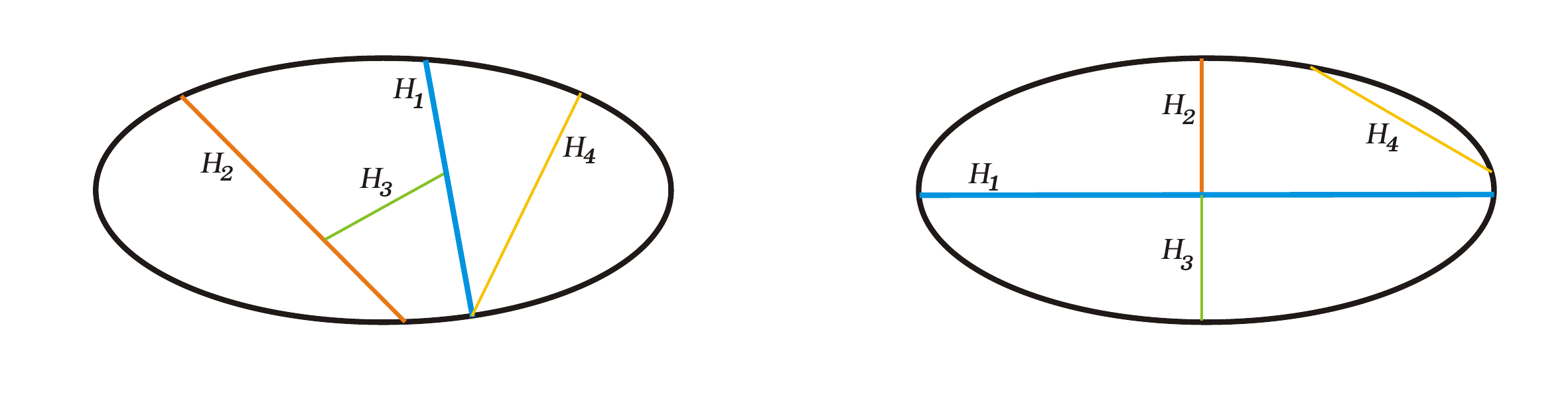}
    \caption{Two 5-divisions of an ellipse, provided by four hiperplane cuts. %(for the second one, the four cuts are given by one horizontal cut, two vertical cuts and an oblique one)
    }
    \label{fig:successive}
\end{figure}

\begin{remark} % este remark se usa tres veces en diferentes demostraciones posteriores.
\label{re:dos}
Note that at least the first  cut in an $n$-division always divides $C$ into two convex regions, which will not necessarily be subsets of the resulting $n$-division (see Figure ~\ref{fig:successive}).
\end{remark}

%­: Note that at least the first  cut in a $n$-division always divide $C$ into two convex regions (which will not necessarily be subsets of the resulting $n$-division, see Figure ~\ref{fig:complete}).}
%\textcolor{blue}{Si no quitamos el remark anterior, podriamos modificar la figura anterior (Fig 1) para que tambien sirviese de ejemplo ilustrativo en este caso. }
%\begin{figure}[ht]
%    \centering
%    \includegraphics[width=0.75\textwidth]{complete}
%    \caption{In these two different 5-divisions of an ellipse $C$, the hyperplane $H$ divides $C$ into two convex regions. %Observe that this does not happen for any other hyperplane cut
%    }
%    \label{fig:complete}
%\end{figure}

Let $F$ denote one of these three classical geometric magnitudes, defined for any compact set in $\rr^d$:
\begin{itemize}
    \item[-] the diameter $D$, which is the largest distance between two points in the set,
    \item[-] the (minimal) width $w$, which is the smallest distance between two parallel supporting hyperplanes of the set, and
    \item[-] the inradius $I$, which is the largest radius of a ball entirely contained in the set.
\end{itemize}

%Let $C$ be a fixed convex body $C$ in $\rr^d$, and let $n\geq2$ be a fixed natural number.

Associated to the magnitude $F$, we consider the following \emph{min-Max type problem}: determine the $n$-divisions $P$ of $C$ that provide the minimal possible value for
$$F(P):=\max\{F(C_1),\ldots,F(C_n)\}, $$
 where $C_1,\ldots,C_n$ are  the subsets given by $P$, as well as finding that value:
\begin{equation}
\label{eq:infimo}
F_n(C)=\inf\{F(P):\ P\ \text{is an}\ n\text{-division of}\ C\}.
\end{equation}

The dual {\em Max-min type problem} seeks for the $n$-divisions $P$ of $C$ for which
$$\widetilde{F}(P):=\min\{F(C_1),\ldots,F(C_n)\}$$
agrees with
\begin{equation}
\label{eq:supremo}
\widetilde{F}_n(C)=\sup\{\widetilde{F}(P):\ P\ \text{is an}\ n\text{-division of}\ C\}.
\end{equation}

Any $n$-division $P$ of $C$ satisfying that $F(P)=F_n(C)$ or $\widetilde{F}(P)=\widetilde{F}_n(C)$ will be called an \emph{optimal $n$-division} of $C$, %for the corresponding problem,
and the values $F_n(C)$ and $\widetilde{F}_n(C)$
%defined in~\eqref{eq:infimo} and~\eqref{eq:supremo},
will be referred to as the \emph{optimal values} of the considered problem.  Additionally, we will say that an $n$-division of $C$ into subsets $C_1,\ldots,C_n$ is \emph{balanced}  if  $F(C_1)=\ldots=F(C_n)$.

The following inequalities are almost straightforward from the previous  definitions:

%The first result concerning the optimal value for our problems is the following:
%\textcolor{blue}{$\to$ aqui habia dos lemas que he fundido en uno, habra que buscar la referencia le:m para cambiarla por le:cotatrivial}

\begin{lemma}
\label{le:cotatrivial}
Let $C$ be a convex body in $\rr^d$. Then, $0< F_n(C)\leq F_m(C)\leq  F(C)$ and $0< \widetilde{F}_n(C)\leq F(C)$, for any  $n\geq m\geq 2$.
\end{lemma}

\begin{proof}
The second chain of inequalities is trivial. For the first one, notice that if $F_n(C)=0$,  using Blaschke selection theorem~\cite[Th.~1.8.7]{schneider} we can find closed subsets $E_1,\ldots,E_n$ of $C$ with $F(E_i)=0$ (and therefore having empty interior), such that $C=E_1\cup\ldots\cup E_n$,   which yields a contradiction, since $C$ has non-empty interior. % \textcolor{blue}{este Ãºltimo resultado dijimos de especificar un poco la demostraciÃ³n, creo que lo he modificado tal y como lo hablamos, pero que lo compruebe Antonio}
Finally, for $m\leq n$, let $P_m$ be an arbitrary $m$-division of $C$ with subsets $C_1,\ldots,C_m$. Without loss of generality, we can assume that $F(P_m)=F(C_1)$. By dividing the subset $C_m$ into $n-m+1$ subsets by succesive hyperplane cuts, we will obtain an $n$-division $P_n$ of $C$ with subsets $C_1,\ldots, C_{m-1}, C_m',\ldots, C_n'$, for which  $F(P_m)=F(P_n)\geq F_n(C)$, and therefore $F_m(C)\geq F_n(C)$.
\end{proof}

\section{min-Max type problems}
\label{sec:mM}
In this section we shall treat the %previously described
min-Max type problems for the diameter, the width and the inradius.
Recall that the optimal value for each of these problems will be denoted by $F_n(C)$, where $F$ stands for the considered magnitude.

\subsection{min-Max problem for the diameter}
For this problem, the precise computation of the corresponding optimal value is, in general, a complicated task, and lower and upper bounds will be provided in  Theorem~\ref{prop:mMDbound2}. Moreover, we will see that the existence of balanced optimal divisions is always assured (see Theorem~\ref{prop:mMDexistence}),
but not all optimal divisions are balanced, as shown in Example~\ref{ex:mMDbalanced}.

%For this problem, the main feature is the estimate of the corresponding optimal value. %For a given planar convex body $C$ in $\rr^d$, the precise computation of $D_n(C)$ is complicated and some lower and upper bounds will be provided (see Theorems~\ref{prop:mMDbound1},~\ref{prop:mMDbound2} and~\ref{prop:mMDbound3}). We note that finding general proper bounds is a hard question, with a strong dependence on the shape of the considered convex body. Moreover, the existence of optimal divisions is always assured (see Theorem~\ref{prop:mMDexistence}), but not all of them will be balanced, as shown in Example~\ref{ex:mMDbalanced}.

\begin{theorem}
\label{prop:mMDexistence}
Let $C$ be a convex body in $\rr^d$. Then, there exists a balanced  optimal $n$-division of $C$ for the min-Max problem for the diameter.
\end{theorem}

\begin{proof}
Let us first prove the existence of optimal divisions. As the optimal value $D_n(C)$ is defined as an infimum, we can consider a sequence  $\{P_k\}$ of $n$-divisions of $C$  such that $\displaystyle{\lim_{k\to\infty}D(P_k)=D_n(C)}$. Denote by  $C_1^k,\ldots, C_n^k$ the subsets given by the division $P_k$, for any  $k\in\mathbb{N}$. Without loss of generality, we can assume that  $D(P_k)=D(C_1^k)$ (consequently,  $D(C_1^k)\geq D(C_j^k)$, for $j=2,\ldots,n$, and  $\displaystyle{ D_n(C)=\lim_{k\to\infty}D(C_1^k)}$).
By applying Blaschke selection theorem~\cite[Th.~1.8.7]{schneider}, for each $j\in\{1,\ldots,n\}$, we have that (a subsequence of) the sequence $\{C_j^k\}$ will converge to a subset $C_j^\infty$, which could have empty interior in some cases. % be a segment in some case.
Therefore,  $C_1^\infty,\ldots,C_n^\infty$ will provide an $m$-division $P^\infty$ of $C$, with $m\leq n$,  %(as some $C_j^\infty$ may reduce to a segment).
satisfying $\displaystyle{D(P^\infty)=D(C_1^\infty)=\lim_{k\to\infty}D(C_1^k)=D_n(C)}$ by construction. If $m=n$, then $P^\infty$ is an optimal $n$-division of $C$, and if $m<n$, we can proceed as in the proof of Lemma~\ref{le:cotatrivial} to obtain an $n$-division $P$ of $C$ (by dividing properly $C_m^\infty$, for instance) such that $D(P)=D(C_1^\infty)=D_n(C)$, thus being optimal.

We will now prove that we can find a \emph{balanced} optimal $n$-division of $C$ by  induction on the number $n$ of subsets.
Let $P$ be an optimal $n$-division of $C$.
If $n=2$, $P$ will be determined by just one hyperplane $H$, with subsets $C_1, C_2$. We can assume that $D(P)=D(C_2)>D(C_1).$ Moreover, we can also assume that $H$ is not parallel to any flat piece of $\partial C$ (if needed, we can consider another optimal division determined by a hyperplane $H'$ close to $H$, with subsets $C_1', C_2'$, satisfying that $C_2'\subset C_2$ and  $D(C_1')<D(C_2')$). % slightly modify $H$ obtaining another hyperplane $H'$,  obtaining $C_2'\subset C_2$ and $D(C_1')<D(C_2')$, and so $D(P_{H'})=D(C_2')\leq D(C_2)$, but this implies that $P_{H'}$ is also optimal).
Let $H^t$ be the hyperplane parallel to  $H$ at distance $t\geq0$ with respect to $C_1$. This yields a new 2-division $P^t$, with subsets $C_1^t, C_2^t$ satisfying $C_1\subseteq C_1^t$, $C_2^t\subseteq C_2$, for each $t\geq0$.
Let $t_1>0$ be the unique value for which $H^{t_1}$ is tangent to $\partial C$, and so $C_1^{t_1}$ equals $C$ and $C_2^{t_1}$ reduces to a single point. Then we have
$$D(C_1^0)=D(C_1)<D(C_2)=D(C_2^0)$$ and
$$D(C_1^{t_1})=D(C)>D(C_2^{t_1})=0,$$
and so, by continuity, there exists $t_0\in(0,t_1)$ such that $D(C_1^{t_0})=D(C_2^{t_0})$. This implies that the corresponding 2-division $P^{t_0}$ of $C$ (determined by $H^{t_0}$) is balanced. Moreover,
$$
D(P^{t_0})=D(C_1^{t_0})=D(C_2^{t_0})\leq D(C_2)=D(P),
$$
where the inequality sign above is due to the set  inclusion $C_2^{t_0}\subset C_2$. Strict inequality above would contradict the optimality of $P$, and so necessarily $D(P^{t_0})=D(P)$ and $P^{t_0}$ is optimal.

Assume now the following induction hypothesis: for a fixed natural number $n\geq2$, there is always a balanced optimal $m$-division of any convex body in $\rr^d$,  for $m<n$. Let us prove that there exists a balanced optimal division for $n$ subsets.
By using Remark~\ref{re:dos}, call  $H$ to a hyperplane cut from $P$ dividing $C$ into two different convex regions $E_1$, $E_2$. As previously, we can assume that $H$ is not parallel to any flat piece of $\partial C$. Let $P_i$ be the division of $E_i$ into $n_i$ subsets induced by $P$, $i=1,2$, with $n_1+n_2=n$.
%Clearly $D(P)=\max\{D(P_1), D(P_2)\}$.
By using the induction hypothesis, let $P_i'$ be a balanced optimal $n_i$-division of $E_i$, $i=1,2$. In particular, \begin{equation}
\label{eq:simplify}
D_{n_i}(E_i)=
D(P_i')\leq D(P_i)\leq\max\{D(P_1), D(P_2)\}=D(P),
\end{equation}
for $i=1,2$.

If $D(P_1')=D(P_2')$, then the divisions $P_1'$, $P_2'$ yield a balanced $n$-division $Q$ of $C$. Moreover, by using~\eqref{eq:simplify}, we have that $D(Q)=D(P_1')\leq D(P)$ and so, in order to avoid a contradiction with the optimality of $P$, it follows that $D(Q)=D(P)$ and then $Q$ is optimal.
On the other hand, if (say) $D(P_1')<D(P_2')$, let  $H^t$ be the hyperplane parallel to $H$ at distance $t\geq0$ with respect to $E_1$, which will divide $C$ into two convex regions  $E_1^t$, $E_2^t$, with $E_1\subseteq E_1^t$ and  $E_2^t\subseteq E_2$. For each $t\geq0$, taking into account the induction hypothesis, we can consider the balanced optimal $n_i$-division $P_i^t$ of $E_i^t$, $i=1,2$. By using a  continuity reasoning similar to the one from the previous case $n=2$, it can be checked that there exists  $t_0>0$
%and two balanced optimal partitions $P_1^{t_0}$ of $E_1^{t_0}$ and  $P_2^{t_0}$ of $E_2^{t_0}$,
such that $D(P_1^{t_0})=D(P_2^{t_0})$. These two divisions will yield a balanced $n$-division $Q$ of $C$. Moreover,
$$
D(Q)=D(P_2^{t_0})=D_{n_2}(E_2^{t_0})\leq D_{n_2}(E_2)\leq D(P),
$$
taking into account the set inclusion $E_2^{t_0}\subset E_2$ and~\eqref{eq:simplify}. Strict inequalities above would contradict the optimality of $P$, and so necessarily $D(Q)=D(P)$ and then $Q$ is optimal.
\end{proof}

The following example shows that not all optimal divisions for this problem are necessarily balanced.

\begin{example}
\label{ex:mMDbalanced}
\textup{For a circle $C$, it is clear that any division $P$ of $C$ into two subsets satisfies that $D(P)=D(C)$, and so $D_2(C)=D(C)$ and any 2-division of $C$ is  optimal (this is related to the classical Borsuk's problem, see~\cite{ac-borsuk}). In particular, any non-balanced 2-division of $C$ will be optimal. The same fact happens for any ball in general dimension}.
% Moreover, these examples also illustrate that, in general, there is no uniqueness for the optimal divisions of this problem.}
\end{example}

We will now focus on the optimal value $D_n(C)$. In general, it is not possible to determine precisely that value, and lower and upper bounds will be established. First, let us introduce the notion of {\em orthogonal widths} for a convex body $C\subset\rr^d$.

In $\rr^2$, any planar convex body $C$ is contained in a 2-orthotope (rectangle) $H_C$ with side lengths $w_1=w(C) \leq w_2$,
% $w^{ort(D(C))}$,
where $w_2$ is the  width of the projection of $C$ on any supporting line of $C$ providing $w(C)$, see Figure~\ref{fig:hyper}.
%(in this work, a width segment of $C$ will be a segment in $C$ of length $w(C)$ and orthogonal to two supporting hyperplanes  determining $w(C)$).
%This rectangle will be the corresponding hyperprism $H_C$ in the planar case,
\begin{figure}[ht]
    \centering
    \includegraphics[width=0.38\textwidth]{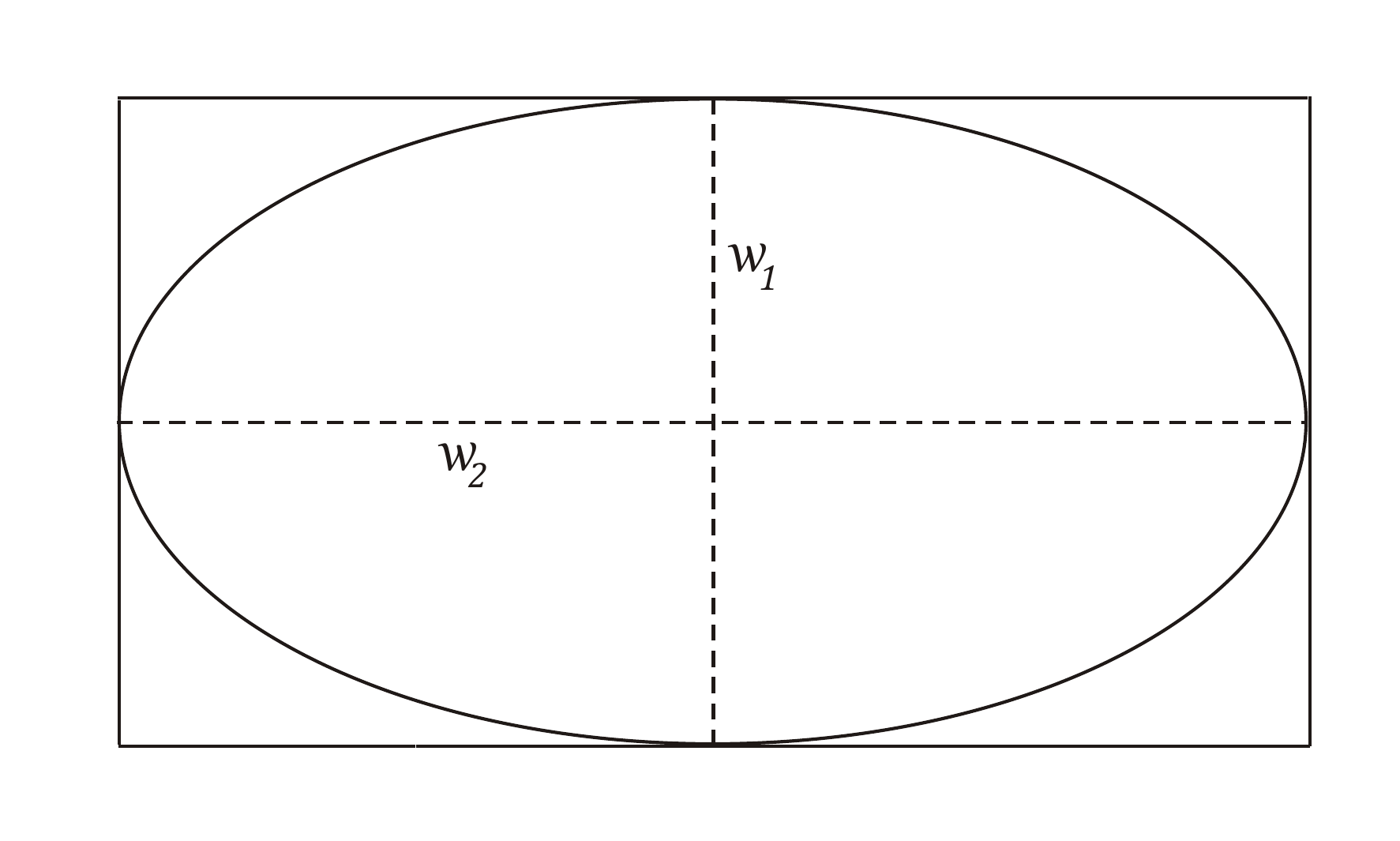}
    \caption{Associated 2-orthotope to an ellipse}
    \label{fig:hyper}
\end{figure}

%In a similar way, in $\rr^3$ we can consider the 3-orthotope $H_C$ with height $w_1=w(C)$, and rectangular basis with edge lengths $w_2$ and $w_3$, where $w_2$ is the minimal width of $C$ among all the directions from the plane $P_1$ orthogonal to a fixed width segment of $C$,
%(namely, $\widetilde{w}^{ort(D(C))}$),and $w_3$ the minimal width of $C$ in the direction of $P_1$  which is orthogonal to $w_2$. It is clear again that $H_C$ contains $C$.

In a similar way, for a convex body $C\subset\rr^d$, it is clear that $C$ is contained in a $d$-orthotope $H_C$ of lengths $w_1\leq w_2\leq \ldots \leq w_d$, where $w_1=w(C_1)$, $C_1=C$, and for $i\geq 2$, $w_i=w(C_i)$ for $C_i=\pi_i(C_{i-1})\subset\rr^{d-i+1}$, where $\pi_i$ denotes the orthogonal projection on any supporting hyperplane of $C_{i-1}$ determining $w(C_{i-1})$.
%This construction can be extended to general dimension, and then we can consider the \emph{hyperprism} $H_C$ in $\rr^d$ with height $w(C)$, and whose basis is a $(d-1)$-polytope with edge lengths $w_2, w_3,\ldots, w_{d}$, where $w_2$ is the minimal width of $C$ among all the directions from the hyperplane $P_1$ orthogonal to a fixed width segment of $C$, $w_3$ is the minimal width among all the directions from the $(d-2)$-plane in $P_1$ which are orthogonal to $w_2$, and so on. By construction, this hyperprism $H_C$ contains $C$ in a more fitting way that the hypercube $Q$ from Theorem~\ref{prop:mMDbound2} does. Call $w_1:=w(C)$.
The values $w_1,\ldots, w_d$ will be referred to as the  \emph{orthogonal widths} of $C$.

%Observe that the faces of the hyperprism $H_C$ will be given by some pairs of parallel hyperplanes $h_i$, $h_i'$, separated by a distance equal to $w_i$, and bounding a certain slab $B_i$, $i=1,\ldots,d$.

\begin{theorem}
\label{prop:mMDbound2}
Let $C\subset\rr^d$ be a convex body with orthogonal widths $w_1,\ldots,w_{d}$,  and  let $n\in\nn$, $n\geq2$.
%Let $D_n(C)$ be the optimal value of $C$ for the min-Max problem for the diameter when considering $n$-divisions.
Then,
\begin{equation}
\label{eq:mMDbound2}
\frac{1}{n} D(C) < D_n(C)\leq\min\bigg\{D(C),\,
%\sqrt{d}\,\frac{D(C)}{\lfloor n^{1/d}\rfloor,},\,
\sqrt{\frac{w_1^2}{a_1^2}+\frac{w_2^2}{a_2^2}+\ldots+\frac{w_{d}^2}{a_{d}^2}}\bigg\},
\end{equation}
for any $a_1\leq\ldots \leq a_d$ natural numbers such that $n\geq a_1\cdot \ldots \cdot a_d$.
\end{theorem}

\begin{proof}

For the left-hand side of \eqref{eq:mMDbound2}, let $P$ be a balanced optimal $n$-division of $C$ with subsets $C_1,\ldots,C_n$, in view of Theorem~\ref{prop:mMDexistence}. Note that $D_n(C)=D(P)=D(C_i)$, $i=1,\ldots,n$. Fix a segment $s$ in $C$ with $\ell(s)=D(C)$, where $\ell$ represents the Euclidean length. We will distinguish  two possibilities. If $s$ is completely contained in a subset $C_j$, then $D_n(C)=D(C_j)=D(C)$ and the statement trivially holds.  Otherwise,
$P$ will divide $s$ into $m$ segments $s_1,\ldots, s_m$, with $2\leq m\leq n$. We can assume that $s_i\subset C_i$, $i=1,\ldots,m$. Then, it follows that $\ell(s_i)\leq D(C_i)$ and, moreover,
\begin{equation}
\label{eq:segment}
D(C)=\ell(s)=\sum_{i=1}^m \ell(s_i)<\sum_{i=1}^m D(C_i)\leq \sum_{i=1}^n D(C_i)= n\,D(P)=n\,D_n(C), \end{equation}
where the strict inequality above holds since it is always  possible to find (at least) two points in one of the subsets given by $P$, say $C_j$, at a distance strictly greater than $\ell(s_j)$ (recall that $P$ is determined by hyperplane cuts and so, if $\ell(s_i)=D(C_i)$ for some $i\in\{1,\ldots,m\}$, then $\ell(s_{i+1})<D(C_{i+1})$, see Figure~\ref{fig:segments}).

\begin{figure}[ht]
    \centering
    \includegraphics[width=0.5\textwidth]{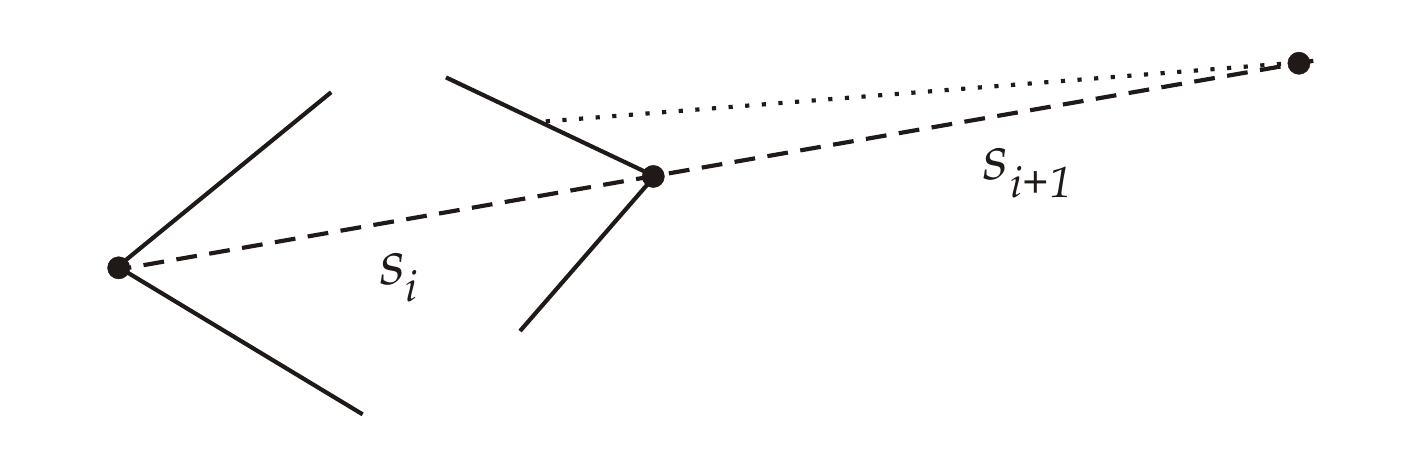}
    \caption{If a segment $s_i$ satisfies that $\ell(s_i)=D(C_i)$, then $\ell(s_{i+1})<D(C_{i+1})$: observe that the length of the dotted segment is greater than $\ell(s_{i+1})$}
    \label{fig:segments}
\end{figure}

For the right-hand side of~\eqref{eq:mMDbound2}, let $H_C$ be the $d$-orthotope containing $C$ associated to the orthogonal widths $w_1,\ldots, w_d$ of $C$.
Observe that the faces of  $H_C$ will be given by some pairs of parallel hyperplanes $h_i$, $h_i'$, separated by a distance equal to $w_i$, and bounding a certain slab $B_i$, $i=1,\ldots,d$.
%Consider the following partition  $P_G$ of $G$ into $n$ subsets. Let $s_i$ be a fixed segment contained in $C$, determining $w_i$, for $i=1,\ldots,d$.
Consider $a_i-1$ hyperplanes parallel to $h_i$ and equally spaced in $B_i$, for  $i=1,\ldots,d$. These hyperplanes yield a mesh-type division $P$ of $H_C$ into $r=a_1\cdot\ldots\cdot a_d$ subsets $H_1,\ldots,H_r$ (which can be seen as a $r$-division given by $r-1$ successive hyperplane cuts), where each $H_i$ is a $d$-orthotope with edge lengths $w_1/a_1$, $w_2/a_2,\ldots,w_{d}/a_{d}$. Then,
$$D(P)=D(H_i)=\sqrt{\frac{w_1^2}{a_1^2}+\frac{w_2^2}{a_2^2}+\ldots+\frac{w_{d}^2}{a_{d}^2}},$$ which constitutes an upper bound for $D_n(C)$ in view of Lemma~\ref{le:cotatrivial} (since $P$ will induce a $m$-division of $C$, with $m\leq r\leq n$, by hypothesis).
The proof finishes taking into account that   $D_n(C)\leq D(C)$, again by Lemma~\ref{le:cotatrivial}.
\end{proof}

\begin{remark}
\label{re:almostsharp}
The lower bound from  Theorem~\ref{prop:mMDbound2} can be considered  sharp in the following sense: for any $\varepsilon>0$, there always exists a convex body $C$ in $\rr^d$ such that $D_n(C)< D(C)/n+\varepsilon$.
It can be checked  that it suffices to take a  narrow enough rectangle in the planar case, or the corresponding  analogous set in higher dimension.
\end{remark}

\begin{example}
For a square $C\subset\rr^2$ of side length $1$ and $n=7$, Theorem~\ref{prop:mMDbound2} above gives $D_7(C)\leq \frac{\sqrt{13}}{6}$, which is obtained for $a_1=2$, $a_2=3$.
\end{example}

\subsection{min-Max problem for the width}
For this problem, we stress that the corresponding optimal value $w_n(C)$  for a given convex body $C$ can be computed:
we will see in Theorem~\ref{prop:mMwov} that such a value equals  $w(C)/n$
(extending~\cite[Le.~4.1]{CG}),
as well as obtaining the existence of optimal divisions. %This easily allows to obtain  optimal divisions of $C$ by partitioning properly any slab which provides $w(C)$, see Remark~\ref{re:mMwexistence}.
Moreover, it also shows that all optimal divisions are balanced in this setting (which improves~\cite[Le.~2.3]{CG}).
Finally, we will  describe a (linear) algorithm for determining the optimal value (and so an optimal division) for any convex polygon.
We start by recalling the following well-known result due to T.~Bang~\cite{bang}. %, which will be used along this work.

\begin{lemma}(\cite{bang})
\label{le:bang}
Let $C$ be a convex body in $\rr^d$. Assume that $\displaystyle{C\subset\bigcup_{i=1}^m B_i}$, where $B_i$ is a slab delimited by two parallel hyperplanes in $\rr^d$, $i=1,\ldots,m$. Then,
$\displaystyle{w(C)\leq\sum_{i=1}^m w(B_i)}$.
\end{lemma}

% XXX Meto en esta proposicion tambien la existencia, comentada en el remark posterior. Y pasa a ser teorema.
\begin{theorem}
\label{prop:mMwov}
Let $C$ be a convex body in $\rr^d$. Then, there exists an optimal $n$-division for the min-Max problem for the width. Moreover,
%Let $w_n(C)$ be the optimal value of $C$ for the min-Max problem for the width when considering $n$-divisions. Then,
\begin{equation}
\label{eq:vow}
w_n(C)=w(C)/n,
\end{equation}
and any optimal $n$-division of $C$ is balanced.
\end{theorem}

\begin{proof}
For any $n$-division $P$ of $C$ into  subsets $C_1,\ldots,C_n$, by applying Lemma~\ref{le:bang} we will have that
$$
w(C)\leq\sum_{i=1}^n w(C_i)\leq n\,w(P),
$$
which implies that $w_n(C)\geq w(C)/n$. On the other hand, let $B$ be a slab providing $w(C)$. We can construct an $n$-division $P_0$ of $C$ by considering $n-1$ parallel hyperplanes equally spaced in $B$, see Figure~\ref{fig:P0}. It is clear that  $w(P_0)=w(C)/n$, which gives that $P_0$ is optimal, as stated.
Finally, let $P$ be an optimal $n$-division of $C$ which is not balanced. If $C_1,\ldots, C_n$ are the subsets of $C$ determined by $P$, then we can assume without loss of generality that $w(P)=w(C_1)>w(C_2)$.  By applying Lemma~\ref{le:bang}, it follows that
$$
w(C)\leq\sum_{i=1}^n w(C_i)<n\,w(C_1)=n\,w(P),
$$
which implies that $w_n(C)=w(P)>w(C)/n$. This contradicts~\eqref{eq:vow}, and so $P$ must be balanced.
\end{proof}

\begin{figure}[ht]
    \centering
    \includegraphics[width=0.75\textwidth]{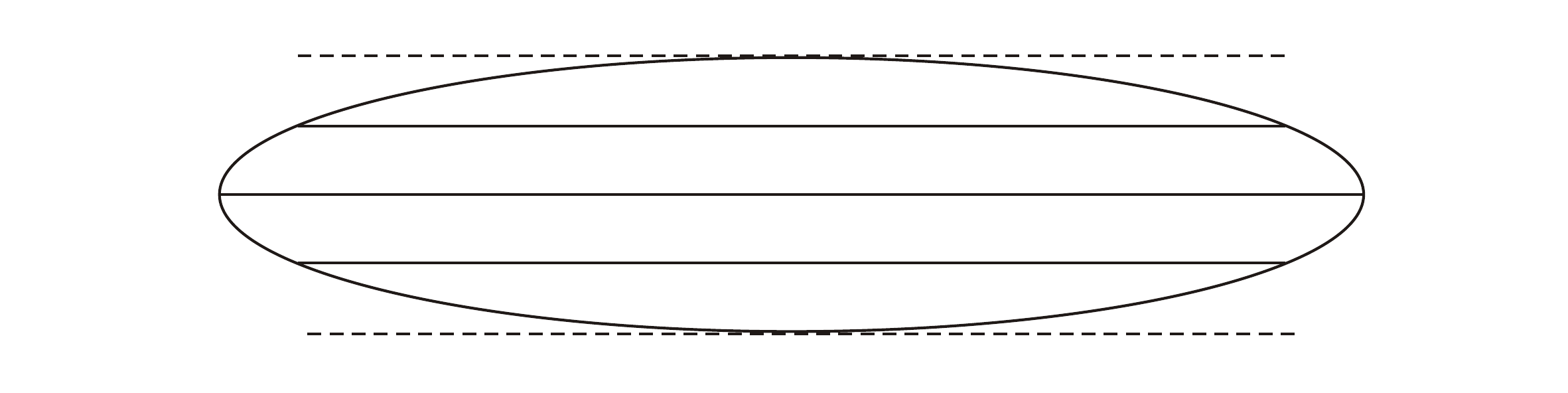}
    \caption{An optimal 4-division of an ellipse% for the min-Max problem for the width
    }
    \label{fig:P0}
\end{figure}

In view of  Theorem~\ref{prop:mMwov}, determining the width of a convex body $C$ immediately leads to the optimal value for the corresponding  min-Max problem, as well as to an optimal $n$-division of $C$ (given by an equally-spaced partition of a slab providing $w(C)$). Such a determination is, in general, a hard question from the computational point of view.
However, for the family of convex polygons, this can be done by means of a \emph{linear} procedure. The key point herein is the following well-known property: for a convex polygon $C$, and a slab $B$ providing $w(C)$, it holds that at least one of the lines delimiting $B$ lies in a side of $C$. Thus, we can calculate the  \emph{relative width} of $C$ with respect to each side of $C$ (by considering parallel lines to each of them), being $w(C)$ the smallest of those values. In computational terms, this is a linear process with respect to the number of sides of $C$, which can be easily implemented.

%As seen in Theorem~\ref{prop:mMwov}, the optimal value for the min-Max problem for the width when considering $n$-divisions of a convex body $C$  equals $w(C)/n$, and an optimal $n$-division of $C$ is given by an equally-spaced partition of any slab providing $w(C)$. Consequently, the determination of the width of $C$ immediately leads to $w_n(C)$ (as well as to a particular optimal division). Such a determination, in general, can be a hard question from the computational point of view. However, for the family of convex polygons, this can be done by means of a \emph{linear} procedure. The key point herein is the following well-known property: given a polygon $C$, and a slab $B$ providing $w(C)$, it holds that at least one of the lines delimiting $B$ lies in a side of $C$. Thus, we can calculate the corresponding \emph{relative width} with respect to each side of $C$ (by considering parallel lines to each of them), being $w(C)$ the smallest of those values. In computational terms, this is a linear process with respect to the number of sides of $C$, which can be easily implemented.

\subsection{min-Max problem for the inradius}
\label{subsec:mMi}
As pointed out in the Introduction, the min-Max problem for the inradius
is known as \emph{Conway's fried potato problem} %in literature~\cite[Problem~C1]{cfg}
and has been deeply studied  in~\cite{bb} (see also~\cite{bb2}).
In that paper, the authors proved that the optimal value of a given convex body $C$ for this problem
can be expressed in terms of the width of a  certain \emph{rounded body} associated to $C$,
as well as describing an optimal division~\cite[Th.~1]{bb}.
These results are stated in Theorem~\ref{prop:bb} for the sake of completeness (see also Definition~\ref{def:rounded}).
Our main contribution in this setting is providing a quadratic  algorithm,
based on the notion of \emph{medial axis}, which leads to the optimal value of any polygon (see Subsection~\ref{subsec:mMi-algorithm}), and the sharp lower bound for that value from Theorem~\ref{prop:desi}.

\begin{definition}
\label{def:rounded}
Let $C$ be a convex body in $\rr^d$, and let $0<\rho\leq I(C)$.
The $\rho$-rounded body $C^\rho$ of $C$ is the union of all the balls of radius $\rho$ which are contained in $C$.
\end{definition}

The notion of rounded body has been previously considered for  different problems in literature. For instance, these sets appear when studying the \emph{isoperimetric problem} inside a convex body of $\rr^d$ (even constituting some solutions in the planar case,  see~\cite{sz}), or when dealing with the classical \emph{Cheeger problem} (the Cheeger set of any planar convex body is characterized as a certain rounded body, see~\cite{kl}).
%Related to our work, A.~Bezdek and K.~Bezdek~\cite{bb} strongly use this notion for investigating the dual min-Max problem for the inradius (see details in Subsection~\ref{subsec:mMi}).
We also note that this construction can be extended to  radius $\rho$ equal to zero, by assuming that $C^0=C$, for any convex body $C$.
% aquí se podría meter una figura mostrando un rounded body concreto, pero es que no veo sitio en el texto donde referenciarlo adecuadamente.

\begin{theorem}(\cite[Th.~1]{bb})
\label{prop:bb}
Let $C$ be a convex body in $\rr^d$. Then, $I_n(C)$ is the unique number $\widetilde{\rho}$ such that
\begin{equation}
\label{eq:bb}
w(C^{\widetilde\rho})=2\,\!n\widetilde\rho.
\end{equation}
Moreover, an optimal balanced $n$-division of $C$ is given by $n-1$ parallel hyperplanes, equally spaced between the two hyperplanes delimiting a slab  which provides   $w(C^{\widetilde\rho})$.
\end{theorem}

%\begin{remark}
%Observe that the previous Theorem~\ref{prop:bb} implicitly gives the existence of optimal divisions for this problem, since it describes a particular one, which is also balanced.
The reader can find two different balanced optimal 3-divisions of an equilateral triangle for this problem  in~\cite[Fig.~1]{bb}. An intriguing open question in this setting is investigating whether any optimal division is necessarily balanced, as it  happens for the min-Max problem for the width (recall Theorem~\ref{prop:mMwov}).
%On the other hand,~\cite[Fig.~1]{bb} depicts two different optimal 3-divisions for an equilateral triangle in this setting (identical to the ones shown in Figure~\ref{fig:equilateral}), so the uniqueness of optimal divisions does not hold for this problem in general.
%\end{remark}

\begin{remark}
For a given convex body $C$, if we set $I_n(C)=\widetilde{\rho}$,  %denotes the optimal value of $C$ for the min-Max problem for the inradius when considering $n$-divisions,
Theorems~\ref{prop:mMwov} and~\ref{prop:bb} yield that
$$
2\widetilde{\rho}=w_n(C^{\widetilde{\rho}}).$$
%where $w_n(C^{\widetilde{\rho}})$ is the optimal value of $C^{\widetilde{\rho}}$ for the   min-Max problem for the width.
In particular, if $C$ is \emph{rounded enough}, the optimal values for the min-Max problems for the width and the inradius will  coincide, up to a constant. However, we point out that the optimal divisions in these two situations
%una situacion es el min-Max del inradio para C; y otra situacion es el min-Max de la anchura para el redondeado de C
will differ, in general: for an equilateral triangle $\mathcal{T}$, the aforementioned~\cite[Fig.~1(b)]{bb} shows an optimal 3-division of $\mathcal{T}$
% and consequently of $\mathcal{T}^{\widetilde{\rho}}$
for the inradius (and so, also optimal for the $I(\mathcal{T})$-rounded body $\mathcal{T}^{I(\mathcal{T})}$  of $\mathcal{T})$, which is not optimal for $\mathcal{T}^{I(\mathcal{T})}$
%he corresponding rounded body of $\mathcal{T}$
when considering the width (in fact, such a 3-division of $\mathcal{T}^{I(\mathcal{T})}$ is not even balanced for the width).
%, see Proposition~\ref{prop:mMwbalanced}).
\end{remark}
% esto último quiere expresar que los valores optimos de min-Max Inradius de C y min-Max Width del C-rounded coinciden (up to a constant), pero las particiones optimas de esas cosas no, en general.
% la 3-division de la figura 1 de bezdek-dezdek es optima para el inradio para el triángulo, y también será óptima para el redondeado del triángulo. Pero tal 3-división no es óptima para el redondeado si consideramos la anchura (primero, porque no es siquiera balanced: el subconjunto que es un triangulillo interior tiene más anchura que los otros subconjuntos, y ya hemos probado que toda optima es equilibrada para tal problema; y segundo, el valor óptimo para el redondeado será la anchura del redondeado dividida por tres, que no es la anchura de los subconjuntos dados por tal 3-division)

Theorem~\ref{prop:desi} gives an explicit lower bound for $I_n(C)$
in terms of the inradius of the considered convex body $C$, by using the following result due to V.~Kadets
(which is, in some sense, analogous to Lemma~\ref{le:bang} by Bang, and originally stated in a more general context).

\begin{lemma}(\cite[Th.~2.1]{kadets})
\label{le:kadets}
Let $C$ be a convex body in $\rr^d$, and let $P$ be an $n$-division of $C$ into subsets $C_1,\ldots,C_n$. Then,
$\displaystyle{I(C)\leq\sum_{i=1}^n I(C_i)}.$
\end{lemma}

\begin{theorem}
\label{prop:desi}
Let $C$ be a convex body in $\rr^d$. %Let $I_n(C)$ be the optimal value of $C$ for the min-Max problem for the inradius when considering $n$-divisions.
Then,
\begin{equation}
\label{eq:desi}
I_n(C)\geq I(C)/n.
\end{equation}
\end{theorem}

\begin{proof}
Let $P$ be an  $n$-division of $C$ with subsets $C_1,\ldots, C_n$.
Lemma~\ref{le:kadets} implies that $$\displaystyle{I(C)\leq\sum_{i=1}^n I(C_i)\leq n\,I(P)},$$
and so $I(P)\geq I(C)/n$. Thus, $I_n(C)\geq I(C)/n$.
\end{proof}

\begin{remark}
Inequality~\eqref{eq:desi} turns into an equality when, for instance, an inball $B$ of $C$ touches $\ptl C$ at exactly two points
(in that case, those two points will be necessarily antipodal in $B$, and one can construct an optimal $n$-division by means of $n-1$ appropriate parallel equispaced hyperplanes).   %orthogonal to the segment joining those two points will provide an $n$-division $P$ such that %described in Remark~\ref{re:alternative} will be balanced, satisfying that $I(P)=I(C)/n$).
\end{remark}

\subsubsection{Algorithm for the optimal value of convex polygons}
\label{subsec:mMi-algorithm}
As commented previously, the  min-Max problem for the inradius has been studied in~\cite{bb}, and the corresponding optimal value is  given  theoretically in Theorem~\ref{prop:bb}, by means of a certain rounded body associated to the original set. However, identifying that specific rounded body  constitutes a hard task in general, since we do not have a systematic method for that purpose. We will now describe a constructive procedure which will lead us to the optimal value for this problem when considering an arbitrary convex polygon. We anticipate that our algorithm turns out to be of \emph{quadratic order} with respect to the number of sides of the polygon.
%and makes use of the notion of \emph{medial axis}.
%(or equivalently, the boundary of a certain  Voronoi diagram, see~\cite{preparata, csw}).
Throughout this Subsection, $d$ will stand for the Euclidean planar distance.

We need some preliminary results. It is well known that the width of a convex polygon is given by a slab delimited by two supporting lines, with at least one of them containing a side of the polygon. Lemma~\ref{le:sup-pol} below  shows that this property extends to any rounded body associated to a convex polygon. %This will be helpful in order to compute its width.

\begin{lemma} \label{le:sup-pol}
Let $C$ be a convex polygon, and let $0<\rho\leq I(C)$. Then, there exists a slab %$B$
providing $w(C^{\rho})$, such that one side of $C$ is contained in the boundary of the slab. %$\ptl B$.
\end{lemma}

\begin{proof}
Recall that $\rho \leq I(C^{\rho}) \leq w(C^{\rho})/2$, and that any tangent line $r$ to $C^{\rho}$ will leave $C^\rho$ entirely contained in one of the half-planes defined by $r$.
%por convexidad
Furthermore, $C^{\rho}$ will be composed in this case by some circular arcs (of radius $\rho$) and some segments (which are pieces of sides of $C$).
%al redondear, ciertos trozos de los lados del poligono engancharán con arcos de radio rho.
Let $B$ be a slab determining $w(C^\rho)$, delimited by two parallel supporting lines $h_1$, $h_2$. Let $x_i$ be a  contact point of $h_i\cap C^{\rho}$, and assume that $x_i$ lies in an arc $\sigma_i$ of  $C^{\rho}$, for $i=1,2$ (otherwise, the statement trivially follows).
% los puntos de contacto en cada linea soporte podrían ser muchos, si la linea y un lado coinciden. Pero en tal caso, ya se tiene lo enunciado.

If $d(x_1,x_2)=w(C^\rho)$,
%where $d$ stands for the Euclidean planar distance,
% en este caso, los dos puntitos estarán justo uno enfrente del otro
let $B'$ be the ball of radius $w(C^{\rho})/2$ passing through $x_1$, $x_2$. Since $\rho\leq w(C^{\rho})/2$, we have two possibilities. In the case of equality, then $\sigma_1$, $\sigma_2$ lie in $\ptl B'$ and we can rotate the lines $h_1$, $h_2$ tangentially along $\ptl B'$ (preserving their parallel character),  obtaining new slabs which also provide $w(C^{\rho})$, until one of the rotating lines meets a segment of $C^{\rho}$, and so we are done.
%(note that this is similar to the application of the rotating calipers technique in Subsection~\ref{subsec:MmD}).
%si B' y los arcos coinciden, podemos ir rotando las lineas soporte, obteniendo nuevas bandas que contienen al redondeado, hasta encontrarme un segmento. Ya tengo ahí una banda dando la anchura con una de sus lineas conteniendo a un lado del polígono.
% this procedure is similar to the one using rotating calipers, in Subsection~\ref{subsec:MmD-algorith}.
On the other hand, in the case of strict inequality, we can proceed analogously: rotate  $h_i$ along $\sigma_i$ (which is now strictly contained in $B'$) to obtain slabs containing $C^{\rho}$ whose width is smaller than $w(C^{\rho})$, yielding a   contradiction, see Figure~\ref{rounded2}.
% Haciendo lo mismo, ahora rotando a lo largo de los arcos, obtendré igualmente bandas que contienen al redondeado. Pero tales bandas ahora tienen anchura menor que $w(c^{\widetilde{\rho}})$, porque la distancia entre las lineas delimitadoras de esas bandas es la longitud de un segmento que está contenido estrictamente en B', siendo tal longitud entonces menor que el diametro de B', que es justamente la anchura del redondeado.
\begin{figure}[ht]
\begin{center}
\includegraphics[width=0.8\textwidth]{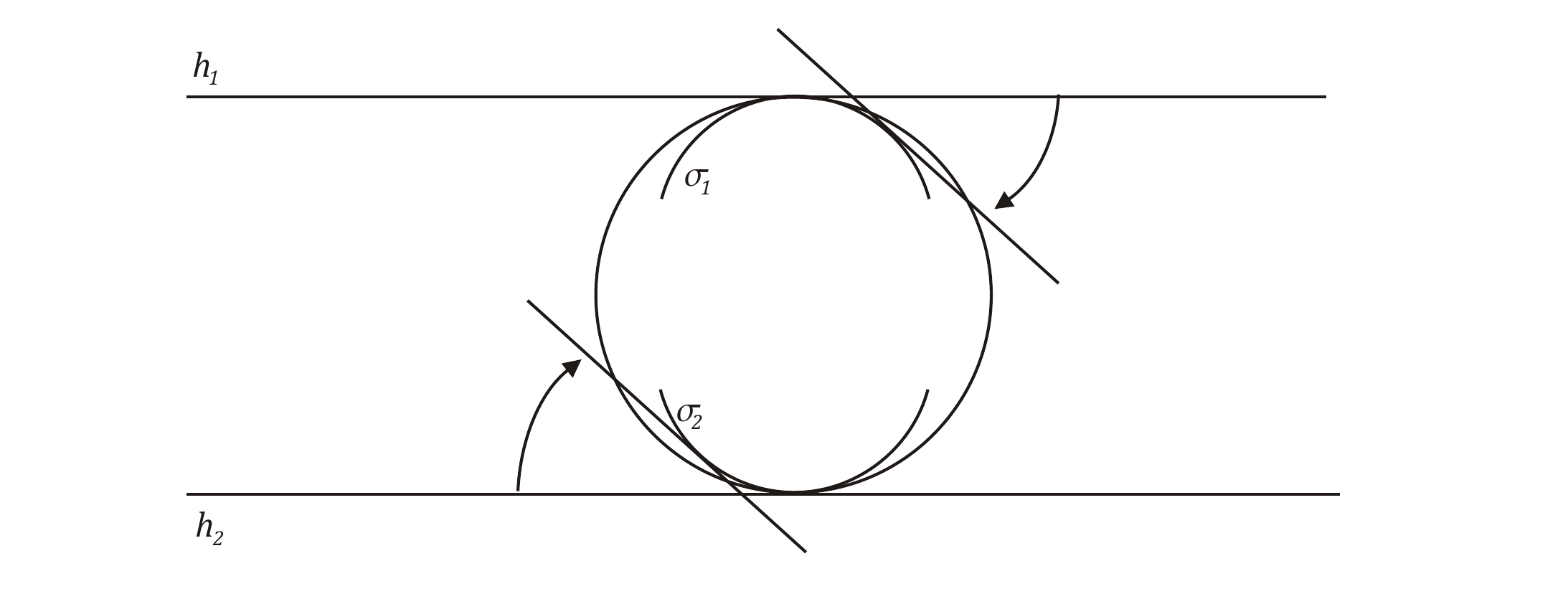}
\end{center}
\caption{If $\rho<w(C^\rho)/2$, a rotation argument yields a slab containing $C^\rho$, with width strictly smaller than $w(C^\rho)$} \label{rounded2}
\end{figure}

If $d(x_1,x_2)>w(C^\rho)$, we can apply a similar reasoning: by rotating $h_i$ along $\sigma_i$ as shown in Figure~\ref{rounded1}, it can be checked that the resulting slabs (which also contain $C^\rho$) have width smaller than $w(C^{\rho})$, which also gives a contradiction.
%ahora hay que rotar en la dirección adecuada, que es la que buscaria acercar los laditos de los arcos "más cercanos", que es la que muestra la figura. Se tiene que la banda tiene anchura menor que la anchura del redondeado por una cuestión métrica, de cuentas: giramos dos rectas paralelas azules, obteniendo las rectas rojas, y si ahora giro las dos rectas rojas con centro en el punto intersección de las rectas azules y rojas de arriba, hasta que la roja y la azul de arriba coincidan,  claramente se obtendrá que la recta roja de abajo queda dentro de la banda original, por lo que la anchura de la banda roja será menor que la anchura de la banda azul (esta explicacion usa los colores de la figura original rounded1.pdf, hecha por Isa; lo azul es lo horizontal, y lo rojo es lo oblicuo).
\begin{figure}[ht]
\begin{center}
\includegraphics[width=0.78\textwidth]{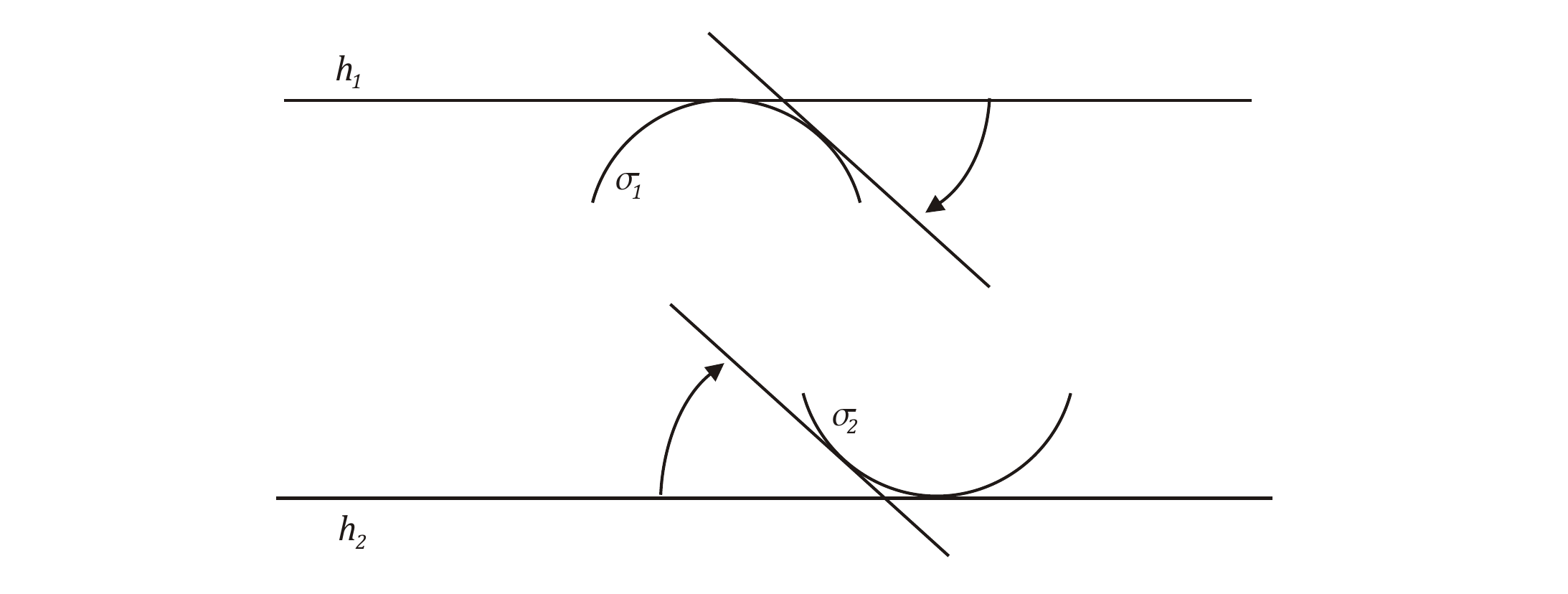}
\end{center}
\caption{If $d(x_1,x_2)=w(C^\rho)$, an analogous rotation argument leads to a slab containing $C^\rho$, whose width is strictly smaller than $w(C^\rho)$} \label{rounded1}
\end{figure}
\end{proof}

In view of Theorem~\ref{prop:bb} and Lemma~\ref{le:sup-pol}, in order to obtain the optimal value of a convex polygon $C$ for this problem, it seems reasonable focusing on each side $L$ of $C$, trying to find the values $\rho>0$ such that $w_L(C^\rho)=2n\rho$, where $w_\sigma(A)$ stands for the width of a planar body $A$ when considering slabs parallel to the direction determined by the line $\sigma$ ($w_\sigma(A)$  will be referred to as the \emph{relative width} of $A$ with respect to $\sigma$).
%, and $n\geq 2$ is a natural number.
It follows that one of those values will be the desired optimal value. This approach will require some new definitions.

Let $C$ be a convex polygon. The \emph{medial axis} $M(C)$ of $C$ is defined as the set of points of $C$ which have more than one closest side of $C$, see Figure~\ref{fig:medialaxis}.
\begin{figure}[ht]
    \centering
    \includegraphics[width=0.3\textwidth]{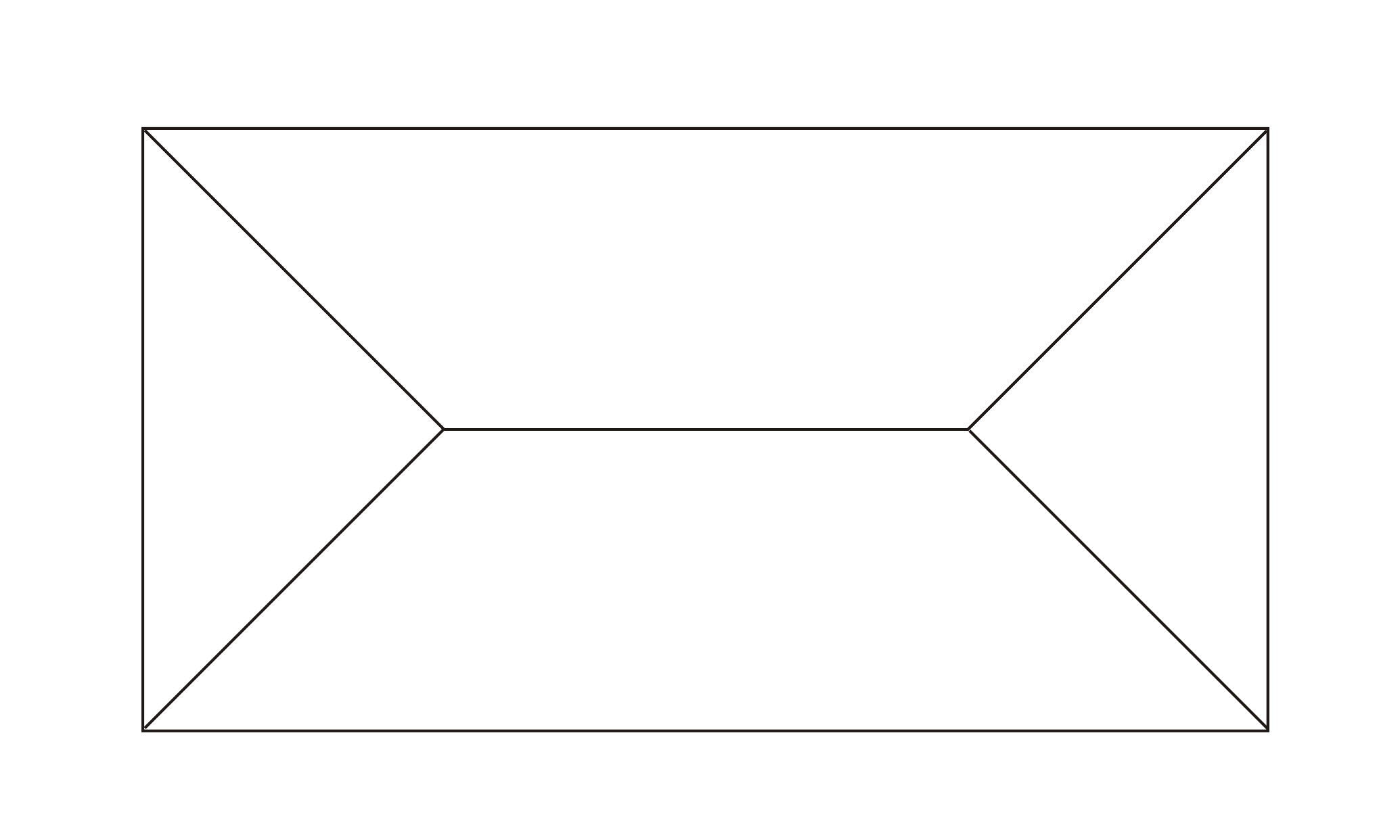}
    \caption{The medial axis of a rectangle}
    \label{fig:medialaxis}
\end{figure}
Equivalently, $M(C)$ is the boundary of the Voronoi diagram associated to $\ptl C$,  and so it will be composed by line segments (it is, in fact, a tree-like graph), see~\cite{preparata, csw}. We point out that each segment of $M(C)$ will be associated to a pair of sides of $C$, and from the computational point of view, it is known that $M(C)$ can be computed in linear time with respect to the number of sides of $C$~\cite[Co.~4.5]{csw}.
%Observe that $M(G)$ provides a partition of $C$ into the subsets of points closest to each edge of $C$.
Moreover, recall that for any $0<\rho\leq I(C)$, the $\rho$-rounded body $C^\rho$ associated to $C$ will be bounded by some circular arcs (of radius $\rho$,  and with centers lying in $M(C)$) and some segments (each of them contained in a side of $C$), by construction. %It follows that the centers of those circular arcs will lie in $M(C)$.

On the other hand, fix a side $L$ of  $C$, and let $L'$ be the supporting line of $C$, parallel to $L$, bounding the slab which provides $w_L(C)$. Any vertex of $C$ contained in $L'$ will be called an \emph{antipodal vertex} to $L$. Note that we can have, at most, two antipodal vertices to $L$, and each of them will have two incident sides of $C$. For each pair of these sides (we can have, at most,  three pairs), the corresponding segment of the medial axis $M(C)$ will be called an  \emph{antipodal segment} to $L$, see Figure~\ref{fig:antipodal}.

\begin{figure}[ht]
    \centering
    \includegraphics[width=0.81\textwidth]{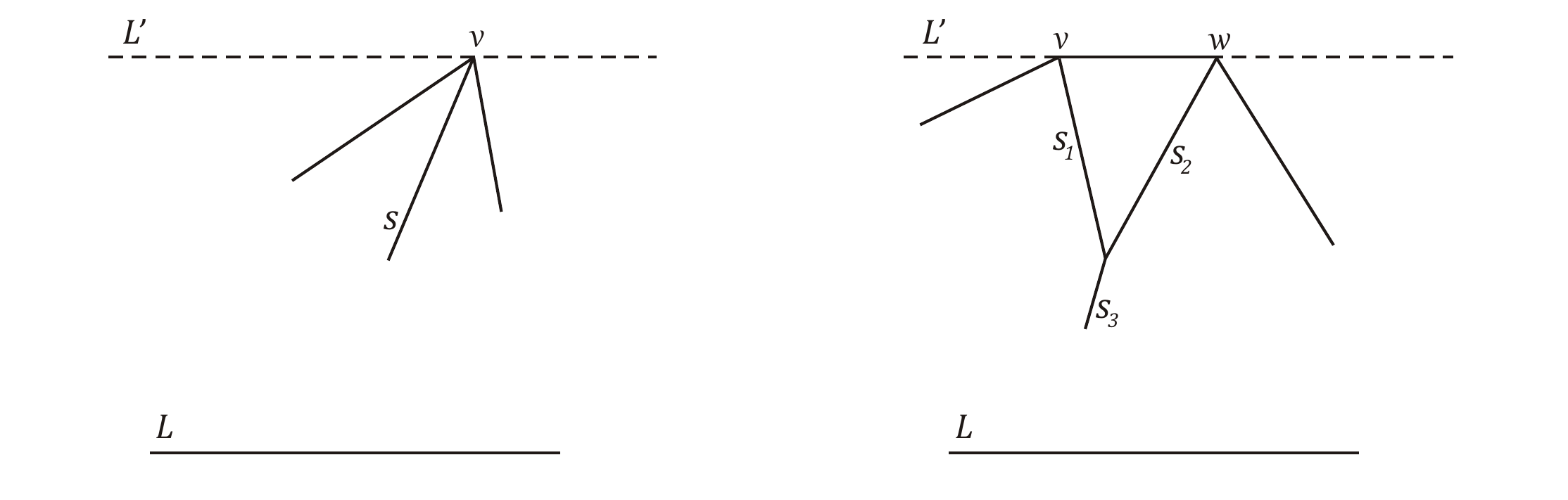}
    \caption{$v$ is an antipodal vertex and $s$ is an antipodal segment to $L$ in the left-hand side figure, and $v$, $w$ are antipodal vertices and $s_1$, $s_2$, $s_3$ are antipodal segments to $L$ in the right-hand side figure}
    \label{fig:antipodal}
\end{figure}

We can now prove the following results on the relative width of a rounded body associated to a convex polygon.

\begin{lemma}
\label{le:alberto}
Let $C$ be a convex polygon, and let $L$ be a side of $C$.
%, and let $s$ be a segment of the medial axis $M(C)$ of $C$. %Assume that $s$ is antipodal to $L$.
Then, there exists a unique value $\rho>0$ such that $w_L(C^\rho)=2n\rho$, where $n$ is a natural number, $n\geq2$.
\end{lemma}

\begin{proof}
%The statement follows due to  the intersection between the graphs of two functions.
Observe that the function $\rho\mapsto 2n\rho$ is linear and strictly increasing, for $0\leq\rho\leq I(C)$.
%, for $0<\rho\geq I(C)$, o incluso incluyendo al zero.
On the other hand, the function $\rho\mapsto w_L(C^\rho)$ is continuous, and decreasing by the set inclusion property for rounded bodies (see also Lemma~\ref{le:alberto2}).
%(, see Le. in fact, it is affine, or piecewise affine piece)
%and decreasing
Moreover, for $\rho_1=0$, it follows that $w_L(C^{\rho_1})=w_L(C)>0$, and for $\rho_2=I(C)$, we have that  $w_L(C^{\rho_2})=2\rho_2<2n\rho_2$. This implies that the graphs of both functions necessarily intersect only once, yielding the statement. \end{proof}

\begin{lemma}
\label{le:alberto2}
Let $C$ be a convex polygon, and let $L$ be a side of $C$. Then, the function $\rho\mapsto w_L(C^\rho)$ is piecewise affine,  for $0<\rho\leq I(C)$.
\end{lemma}

\begin{proof}
Assume firstly that $L$ has only one antipodal vertex $O'$, belonging to sides $L_1$ and $L_2$, and call $s$ to the antipodal segment to $L$ (which is a segment of the medial axis $M(C)$ of $C$). Recall that, for any $\rho>0$, $C^\rho$ is composed by some segments and some circular arcs, whose centers lie in $M(C)$. %Therefore,  $w_L(C^\rho)<w_L(C)$.
For the computation of $w_L(C^\rho)$ for small $\rho$, we only have to focus on $L$ and on the circular arc $\alpha$ of radius $\rho$ centered at a certain point of $s$ (note that $w_L(C^\rho)$ will be given  by the slab delimited by $L$ and $L_\rho$, being $L_\rho$ the tangent line to $\alpha$ parallel to $L$, see Figure~\ref{fig:alberto1}).
\begin{figure}[ht]
    \centering
    \includegraphics[width=0.38\textwidth]{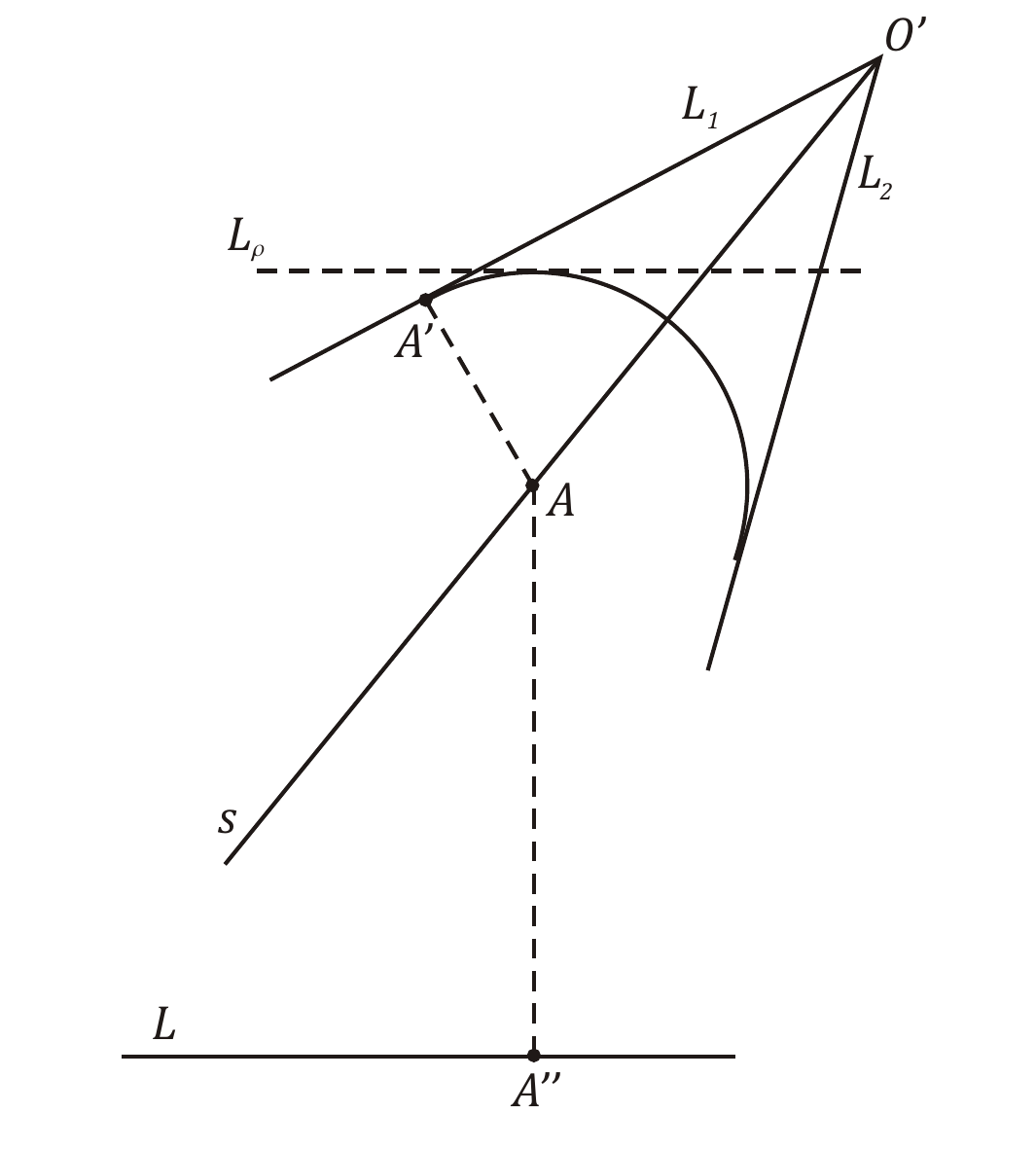}
    \caption{Computing $w_L(C^\rho)$}
    \label{fig:alberto1}
\end{figure}
%since $w_L(C^\rho)$ will coincide with the width of the slab delimited by $L$ and $L_\rho$, where $L_\rho$ is a line touching tangentially $\alpha$.
Taking this into account, an expression for $w_L(C^\rho)$ can be obtained by using the following geometric approach.

Let $A$ be a point in $s$. It is clear that it will  determine a unique rounded body $C^\rho$ associated to $C$, where $\rho=d(A,L_1)$.
%,  and $d$ represents the Euclidean planar distance.
%by considering a proper circular arc centered at $A$ with radius equal to $d(A,L_1)$. Note that this radius will be the rounding radius $\rho$ of the corresponding rounded body.
Call $A'$ to the point in $L_1$ such that $d(A,L_1)=d(A,A')$, and call $A''$ to the orthogonal projection of $A$ on $L$, as Figure~\ref{fig:alberto1} shows.

Then,
\begin{equation}
\label{eq:alberto}
w_L(C^\rho)=d(A'',L_\rho)=d(A,A'')+\rho=d(A,A'')+d(A,A').
\end{equation}
We are going to see that~\eqref{eq:alberto} can be written as an affine expression with respect to $d(A,A')$.

Let $O$ be the intersection of the extensions of $s$ and $L$, and fix a point $B$ in $s$ (different from $A$). By  considering the previous geometric construction for $B$, with two  associated points $B'$ and $B''$, it can be checked that
$$
\frac{d(O',A)}{d(O',B)}=\frac{d(O',A')}{d(O',B')}=\frac{d(A,A')}{d(B,B')}
$$
and
$$
\frac{d(O,A)}{d(O,B)}=\frac{d(O,A'')}{d(O,B'')}=\frac{d(A,A'')}{d(B,B''),}
$$
since the corresponding pairs of triangles are equivalent. Therefore,
$d(A,A'')=\lambda\,(d(O,O')-d(O',A)$ and $d(A,A')=\mu\,d(O',A)$, for certain $\lambda$, $\mu>0$, which implies that
\begin{equation}
\label{eq:alberto2}
d(A,A'')+d(A,A')=(\mu-\lambda)\,d(O',A)+\lambda\,d(O,O'),
\end{equation}
which is an affine expression with respect to  $d(O',A)$. Finally, since $d(O',A)=d(A,A')/\mu$, it turns that~\eqref{eq:alberto2} is an affine expression with respect to $d(A,A')$.
The conclusion is that the value $w_L(C^\rho)$ can be expressed as an affine function on $\rho$, as desired.

The computation of $w_L(C^\rho)$ for larger $\rho$ will continue by considering the segments of $M(C)$ adjacent to the segment $s$ (and possibly the subsequent ones), until $\rho$ goes from zero to $I(C)$. For each of these segments, the corresponding  computation can be done analogously, extending the segments and some sides of the polygon if necessary. Therefore, each considered segment of $M(C)$ will lead to a piece of the function $\rho\mapsto w_L(C^\rho)$ which will be affine, by the former argument, and so, that function will be piecewise affine.

Finally, if $L$ has two antipodal vertices, it can be checked that $\rho\mapsto w_L(C^\rho)$ will be constant at the beginning (because $w_L(C^\rho)=w_L(C)$, for small $\rho>0$), and piecewise affine afterwards by reasoning as in the previous case.
\end{proof}

Taking into account Lemma~\ref{le:alberto},
%and Lemma~\ref{le:alberto2},
each side $L$ of the convex polygon $C$ will provide %(at most)
a value $\rho_L>0$ satisfying $w_L(C^{\rho_L})=2n\rho_L$.
%certain analytical condition.
When applying this to all the sides of $C$,
we will obtain a finite collection $\mathcal{C}$ of such values, which will constitute the candidates for being the optimal value $\widetilde{\rho}$ from Theorem~\ref{prop:bb} (recall that $w(C^{\widetilde{\rho}})=w_L(C^{\widetilde{\rho}})$ for a certain side $L$ of $C$ by Lemma~\ref{le:sup-pol}, which implies that  $\widetilde{\rho}$ must be one of the values in $\mathcal{C}$). Denote by $\rho_1$ the smallest value in $\mathcal{C}$, associated to a side $L_1$ of $C$. We claim that $w(C^{\rho_1})=w_{L_1}(C^{\rho_1})$. On the one hand, $w(C^{\rho_1})\leq w_{L_1}(C^{\rho_1})$ by definition. And on the other hand, by using Lemma~\ref{le:sup-pol} we have that there exists a side $L'$ of $C$ with associated value $\rho'\in\mathcal{C}$ such that
$$
w(C^{\rho_1})=w_{L'}(C^{\rho_1})\geq w_{L'}(C^{\rho'})=2n\rho'\geq 2n\rho_1= w_{L_1}(C^{\rho_1}),
$$
where we are using that $\rho_1\leq\rho'$, and so  $C^{\rho'}\subseteq C^{\rho_1}$. Thus, $w(C^{\rho_1})=w_{L_1}(C^{\rho_1})=2n\rho_1$, and so $\rho_1$ must be the desired optimal value, by the uniqueness property from Theorem~\ref{prop:bb}.
% the optimal value must appear several times in $mathcal{C}$, since several vertices are rounded for $C^\{widetilde{\rho}}$.

All this can be implemented as an algorithm, which is of  quadratic order with respect to the number of sides of the considered polygon.

\begin{theorem}
\label{prop:mMi-algorithm}
Let $C$ be a convex polygon. Then, the optimal value for the min-Max problem for the inradius when considering $n$-divisions can be computed in quadratic time with respect to the number of sides of $C$.
\end{theorem}

\begin{proof}
It can be checked  that all the steps of the preceding procedure can be done, at most, in quadratic time with respect to the number of sides of $C$. From the computational point of view, finding the medial axis of $C$ is a  linear question~\cite[Co.~4.5]{csw}. Moreover, for each side $L$ of $C$, determining the value such that $w_L(C^\rho)=2n\rho$ is also linear, since the graphs of the functions $\rho\mapsto 2n\rho$ and $\rho\mapsto w_L(C^\rho)$ can be depicted  in linear time (for the graph of the second function, we may have to consider a linear quantity of segments of the medial axis of $C$, in view of Lemma~\ref{le:alberto2}).  Therefore, when applying this to all the sides of $C$, it turns out that finding the candidates for being the optimal value (the ones contained in the finite  collection $\mathcal{C}$) will require quadratic time. And finally, notice that choosing the smallest of those candidates can be done in linear time.
\end{proof}

\section{Max-min type problems}
\label{sec:Mm}

In this Section we will treat the corresponding Max-min problems for the diameter, the width and the inradius.
We recall that the optimal value for these problems will be denoted by $\widetilde{F}_n(C)$, where $F$ represents each of those magnitudes.

\subsection{Max-min problem for the diameter}
\label{subsec:MmD}
In this setting, we remark that the optimal value
can be explicitly computed,
being equal to the diameter of the original convex body, as well as the fact that any optimal division is necessarily balanced, see Theorem~\ref{prop:MmDov}.
However, the existence of optimal divisions is not assured in the planar setting
(as we will see in Remark~\ref{re:triangle}), since it depends strongly on the location of the diameter segments of the set,
see Definition~\ref{def:ds} below. %Finally, for any convex polygon $C$, we will describe an algorithm (based on the \emph{rotating calipers} method~\cite{shamos}) for finding the diameter and the diameter segments of $C$, which is linear with respect to the number of sides of $C$
%(this will allow to decide whether an optimal division exists, as well as to construct it if so).

\begin{definition}
\label{def:ds}
Let $C$ be a convex body in $\rr^d$. %By compactness, it follows that the diameter $D(C)$ of $C$ will be attained.
Any segment in $C$ with length equal to $D(C)$ will be called a \emph{diameter segment} of $C$.
\end{definition}

\begin{remark}
\label{re:types}
In the planar case, it is well-known that any pair of diameter segments of a convex body $C$ will necessarily intersect at one point, and it is not difficult to check that such an intersection point will be either an endpoint of both segments or an inner point of both segments (that is, the intersection point cannot be an endpoint of one segment and an inner point of the other segment). % (this would imply that we can find two points of $C$ at a distance greater than $D(C)$).
This implies that any set $J_C$ of diameter segments of $C$ \emph{with disjoint interiors} will consist of one of these two possibilities: either $J_C$ is an equilateral triangle (and so $J_C$ will be called \emph{triangle-type}), or all the diameter segments of $J_C$ share an endpoint (and so  $J_C$ will be called  \emph{fan-type}). Notice that $C$ can have two  maximal sets $J_C$,  $J'_C$ of diameter segments with disjoint interiors, where $J_C$ is triangle-type and $J'_C$ is fan-type (see Figure \ref{fig:differenttypes}).

\begin{figure}[ht]
    \centering
    \includegraphics[width=0.75\textwidth]{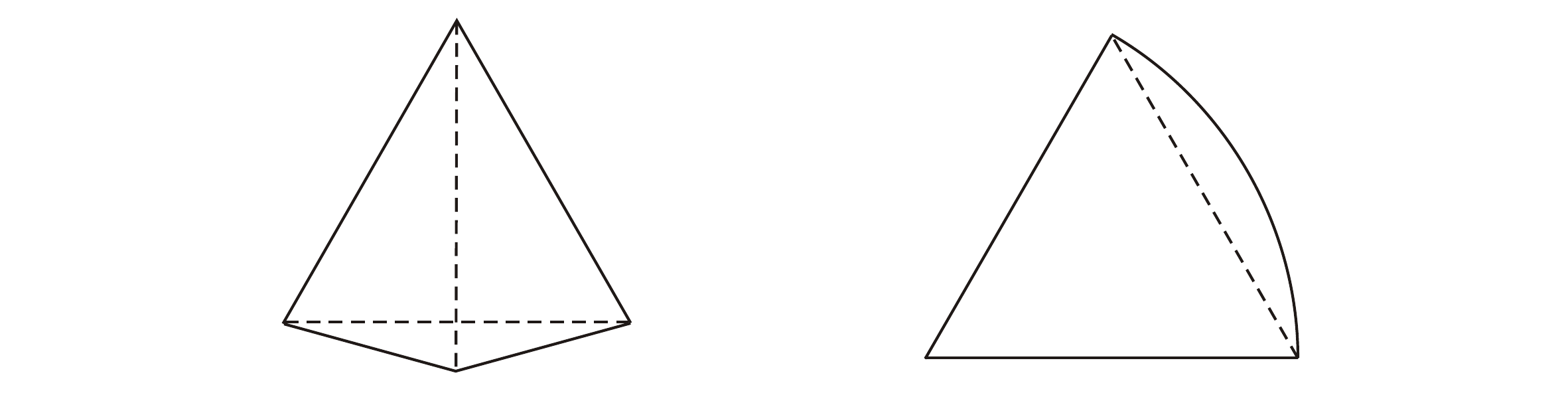}
    \caption{Two  planar convex bodies with different type possibilities for the set of diameter segments with disjoint interiors}
    \label{fig:differenttypes}
\end{figure}

\end{remark}

\begin{theorem}
\label{prop:MmDov} Let $C$ be a convex body in $\rr^d$.
%Let $D_n(C)$ be the optimal value of $C$ for the Max-min problem for the diameter when considering $n$-divisions.
Then, $$\widetilde{D}_n(C)=D(C).$$
Moreover, any optimal $n$-division of $C$ for the Max-min problem for the diameter is balanced.
\end{theorem}

\begin{proof}
Let $s$ be a diameter segment of $C$, and fix a hyperplane $H_0$ in $\rr^d$ containing $s$.
We can consider $n-2$ hyperplanes $H_1,\ldots,H_{n-2}$ parallel to $H_0$
providing (together with $H_0$) an $n$-division of $C$.
By approaching each hyperplane $H_i$ to $H_0$ parallelly, for $i=1,\ldots, n-2$,
we will have a sequence $\{P_m\}$ of $n$-divisions of $C$, each of them with subsets $C_1^m,\ldots,C_n^m$,
satisfying that $\displaystyle{\lim_{m\to\infty} \widetilde{D}(P_m) =D(C)}$,
as each subset $C_i^m$ tends to contain $s$ when $m$ tends to infinity.
Hence $\widetilde{D}_n(C)\geq D(C)$, and by using Lemma~\ref{le:cotatrivial} we conclude that $\widetilde{D}_n(C)=D(C)$.

On the other hand, let $P$ be an optimal $n$-division of $C$ into subsets $C_1,\ldots,C_n$.
Then $\widetilde{D}(P)=\widetilde{D}_n(C)=D(C)$, and so $D(C_i)\geq D(C)$, for $i=1,\ldots,n$.
% since $D(P)=\min\{D(C_1),\ldots,D(C_n)\}$
However, by  direct set inclusion we also have that $D(C_i)\leq D(C)$, which leads to $D(C_i)=D(C)$,
for  $i=1,\ldots,n$, yielding the statement.
\end{proof}

%Theorem~\ref{prop:MmDov} implies that any optimal division, if it exists, is necessarily balanced, as stated in the following result.

% he metido esta proposición junto con la anterior.

%\begin{proposition}
%\label{prop:MmDbalanced}
%Let $C$ be a convex body in $\rr^d$, and let $P$ be an optimal $n$-division of $C$  for the Max-min problem for the diameter. Then, $P$ is balanced.
%\end{proposition}
%
%\begin{proof}
%If $P$ is an optimal $n$-division of $C$ into subsets $C_1,\ldots,C_n$, we will have that
%$D(P)=D_n(C)=D(C)$, in view of Theorem~\ref{prop:MmDov}, and so $D(C_i)\geq D(C)$, for $i=1,\ldots,n$.
%% since $D(P)=\min\{D(C_1),\ldots,D(C_n)\}$
%However, by  direct inclusion we will also have that $D(C_i)\leq D(C)$, which leads to $D(C_i)=D(C)$, for  $i=1,\ldots,n$, yielding the statement.
%\end{proof}

Theorem~\ref{prop:MmDov} also allows to study the existence of optimal divisions for this problem. Theorem~\ref{prop:MmDexistencia1} easily proves this existence in $\rr^d$ when $d\geq 3$, while this does not always occur for $d=2$, as stated in  Theorem~\ref{prop:MmDexistencia2}.

\begin{theorem}
\label{prop:MmDexistencia1}
Let $C$ be a convex body in $\rr^d$, where $d\geq 3$. Then, there exists an optimal $n$-division of $C$ for the Max-min problem for the diameter.
\end{theorem}

\begin{proof}
Let $s$ be a diameter segment of $C$. We can consider $n-1$ distinct hyperplanes containing $s$, yielding an $n$-division $P$ of $C$ into subsets $C_1,\ldots,C_n$. As each subset $C_i$ contains the segment  $s$, it follows that $D(C_i)=D(C)$, for  $i=1,\ldots,n$. Then $\widetilde{D}(P)=D(C)$, %=D_n(C)$,
and so $P$ is optimal, in view of Theorem~\ref{prop:MmDov}.
\end{proof}

\begin{remark}
Observe that the reasoning from the proof of Theorem~\ref{prop:MmDexistencia1} cannot be applied in the planar case, since we will only have \emph{one hyperplane} (which will be a line in this case) containing any  fixed segment in $\rr^2$.
\end{remark}

In order to study the existence of optimal divisions for a planar convex body $C$, we have to make some previous considerations. We already know from Theorem~\ref{prop:MmDov} that all the subsets of an optimal division will have diameter equal to $D(C)$, which implies that all of them will necessarily contain a diameter segment of $C$. This suggests that we will need \emph{enough} diameter segments in $C$ for constructing an optimal division (if we have few diameter segments in $C$, we will not be able to partition $C$ into many subsets with diameter equal to $D(C)$). Therefore, in this planar setting, the existence of optimal divisions will strongly depend on the number of diameter segments of $C$, and more precisely, on how they are placed in $C$. In general, in order to construct an optimal division of $C$, each diameter segment of $C$ contained in the interior of $C$ will lead us to two subsets of $C$ with diameter equal to $D(C)$, by means of appropriate lines, which are the hyperplane cuts in this planar situation (some of these cuts will be determined by the diameter segments, and the other ones will be done \emph{between} the previous cuts, see Figure~\ref{fig:BIcuts1}).
\begin{figure}[ht]
    \centering
    \includegraphics[width=0.75\textwidth]{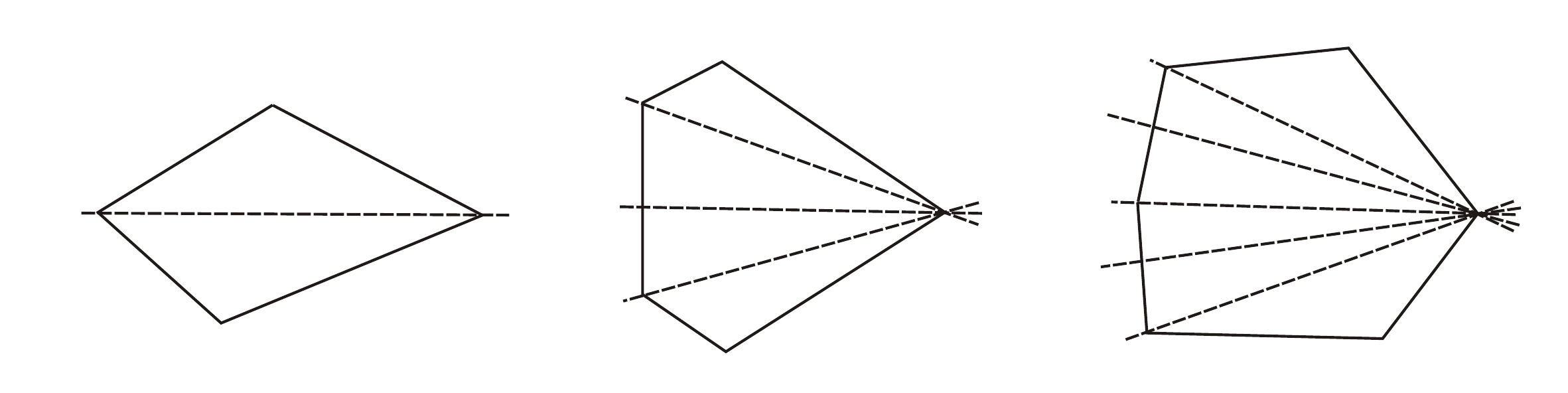}
    \caption{Each  diameter segment in the interior of $C$ leads to two different subsets  (delimited by the dashed lines) with diameter equal to $D(C)$}
    \label{fig:BIcuts1}
\end{figure}
And regarding the diameter segments of $C$ contained in $\partial C$, each of them can only belong to one subset with diameter equal to $D(C)$, see Figure~\ref{fig:BIcuts2}. \begin{figure}[ht]
    \centering
    \includegraphics[width=0.2\textwidth]{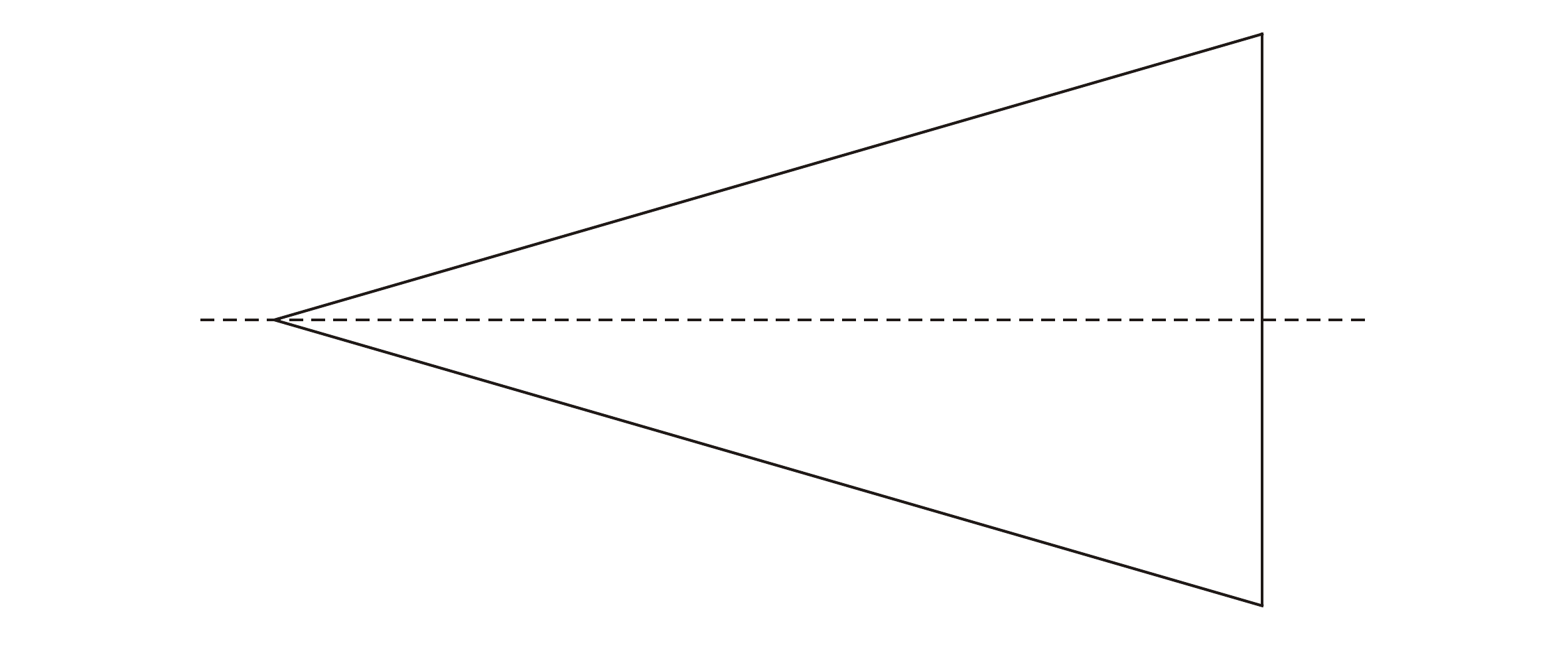}
    \caption{Any diameter segment contained in $\partial C$ will belong to a unique subset (determined by the dashed line) with diameter equal to $D(C)$}
    \label{fig:BIcuts2}
\end{figure}
Taking into account this, we have the following existence result in the planar case, where previous  Remark~\ref{re:types} will be used.

\begin{theorem}
\label{prop:MmDexistencia2}
Let $C$ be a convex body in $\rr^2$. Then, there exists an optimal $n$-division of $C$ for the Max-min problem for the diameter if and only if there exists a set $J_C$ of diameter segments of $C$ with disjoint interiors such that
\begin{equation}
\label{eq:n}
n\leq 2a+b-\delta_C,
\end{equation}
where $a$ is the number of diameter segments of $J_C$ contained in the interior of $C$ and $b$ is the number of diameter segments of $J_C$ contained in $\partial C$, with $\delta_C=1$ if $J_C$ is triangle-type, or $\delta_C=0$ if $J_C$ is fan-type (see Remark \ref{re:types}).
\end{theorem}

\begin{proof}
Assume firstly that \eqref{eq:n} holds, and  we will prove that we can construct an optimal $n$-division of $C$. On the one hand, if $J_C$ is fan-type, we will have that $b\leq 2$ and we can proceed as follows: if $a=0$, then necessarily $b=2$, which implies that $n=2$ and any hyperplane cut \emph{between} the two boundary diameter segments will provide an optimal 2-division of $C$. Otherwise, if $a>0$,
each diameter segment in the interior of $C$ will determine a hyperplane cut, and each boundary diameter segment will give an additional hyperplane cut (placed between the boundary diameter segment and the adjacent interior diameter segment).
Finally, for each two consecutive interior ones, we can consider the corresponding bisector as a new hyperplane cut. This procedure gives $2a-1+b$ hyperplane cuts, yielding a division $P$ of $C$ into at most $2a+b$ subsets, all of them with diameter equal to $D(C)$. Then, $P$ is optimal.
On the other hand, if $J_C$ is triangle-type, it follows that we can  have four different possibilities, depending on the number of diameter segments of $J_C$ in the interior of $C$. For each possibility we can find an optimal division into  at most $2a+b-1$ subsets, as shown in  Figure~\ref{fig:BIcuts3}.

\begin{figure}[ht]
    \centering
    \includegraphics[width=0.93\textwidth]{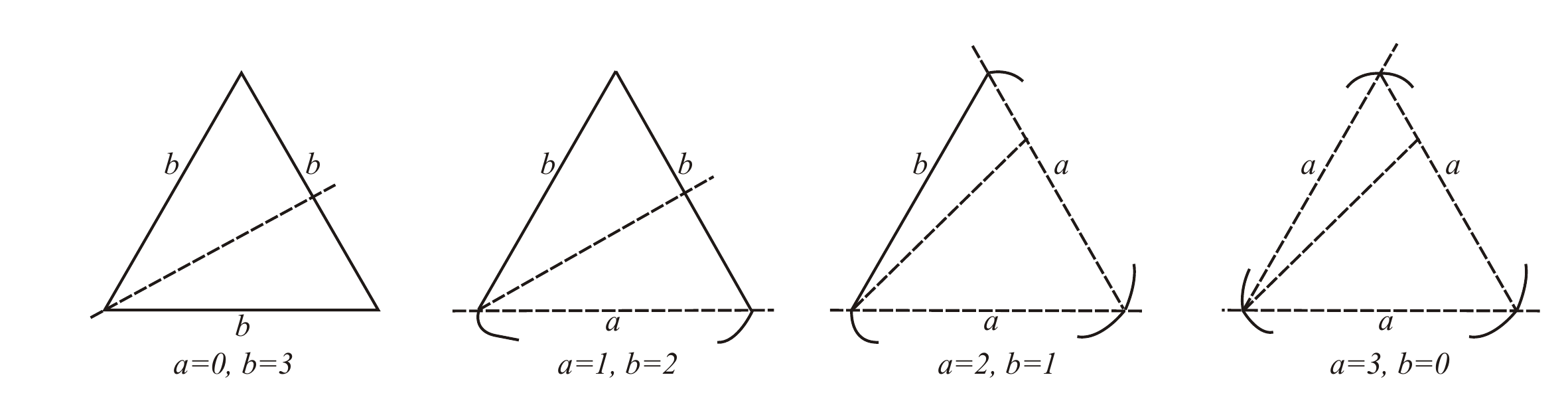}
    \caption{Optimal divisions of $C$ into $2a+b-1$ subsets when $J_C$ is triangle-type (the dashed lines indicate the hyperplane cuts of each division)}
    \label{fig:BIcuts3}
\end{figure}

% implicación existence=>desigualdad:
Conversely, given an optimal $n$-division $P$ of $C$, let $J_C$ be the set composed by one diameter segment from each subset of $P$. The same reasoning as before yields that the maximum number of subsets in $P$ will be  $2a+b-\delta_C$, which finishes the proof.

%, the previous procedure implies that it is not possible to divide $C$ into $n$ subsets, all of them with diameters equal to $D(C)$ (we will not  have enough diameter segments), and so there is  non-existence of optimal divisions in this case, see  Remark~\ref{re:triangle} below.
\end{proof}

% la referencia a este remark muestra que la segunda implicación (existe optima => desigualdad; o equivalentemente, no desigualdad => no existencia de óptima)

%Este remark es un ejemplo
\begin{example}
\label{re:triangle}
Just to illustrate one particular situation where we do not have existence of optimal divisions, consider the following examples: for an equilateral triangle $\mathcal{T}$, we have that the unique maximal set $J_\mathcal{T}$ of diameter segments of $\mathcal{T}$ with disjoint interiors will be composed by the three sides of $\mathcal{T}$. Hence, $a=0$, $b=3$ and $\delta_{\mathcal{T}}=1$, and so $2a+b-\delta_{\mathcal{T}}=2$. This means that, for any $n\geq 3$, there are no optimal $n$-divisions of $\mathcal{T}$ for  the Max-min problem for the diameter (that is, any division of $\mathcal{T}$ into three or more subsets will always have a subset with diameter strictly smaller than $D(\mathcal{T})$). Another simple example is given by a circle $C$, for which $J_C$ is composed by just one interior diameter segment and so $2a+b-\delta_C=2$. It is clear that any $n$-division of $C$, with $n\geq3$, will necessarily have a subset with diameter strictly smaller than $D(C)$, yielding the non-existence of optimal divisions of $C$ in these cases.
\end{example}

We have seen in Theorem~\ref{prop:MmDov} that the optimal value for the Max-min problem for the diameter is precisely the diameter of the considered convex body. Although the computation of the diameter of a general convex body is a hard task, if we focus on the family of convex polygons, that value can be obtained by using a linear algorithm due to M.~I.~Shamos~\cite{shamos} (see also~\cite{toussaint}), which is based on the so-called technique of \emph{rotating calipers}.
%We will now give a brief description of that algorithm, for sake of completeness.
%Recall that the diameter of an arbitrary convex body coincides with the greastest distance between two supporting hyperplanes.
%Let $C$ be a convex polygon determined by $m$ vertices $p_1,\ldots, p_m$.  %(and so with $m$ edges).Fix, for instance, the horizontal direction, and consider the two horizontal parallel supporting lines of $C$. These two lines will touch $\partial C$ at some points, providing a pair of vertices of $C$ (or, at most, four different pairs, if two sides of $C$ are horizontal). By rotating simultaneously both lines in the clock-wise sense until one of them lie on a different side, we will obtain a new pair of vertices (or, at most, four different pairs). Repeating this process until reaching to the initial horizontal direction, we will have at most $4 m$ pairs of vertices, being the diameter of $C$ the largest distance between the points of a pair. Additionally, the diameter segments of $C$, given by the largest-distance pairs, will be also determined. It is proved in~\cite{shamos}  that this is a linear procedure.Finally, we note that once all the diameter segments of $C$ are identified, we can apply the ideas from Theorem~\ref{prop:MmDexistencia2} in order to obtain an optimal division of $C$ for this problem, provided the existence is guaranteed.
A priori, it seems difficult to extend the previous algorithm to general convex bodies.
Although any planar convex body $C$ can be approximated by a sequence $\{P_k\}$ of polygons,
and the sequence $\{\widetilde{D}(P_k)\}$ will tend to $D(C)$,
the construction of $\{P_k\}$ and the pass to the limit constitute issues which are not controlled from the computational point of view.
On the other hand, the extension to higher dimensions cannot be applied by using the rotating calipers technique,
since we have not a unique direction for the corresponding rotations applied throughout the procedure.

\subsection{Max-min problem for the width}
\label{subsec:Mmw}
For this problem, Theorem~\ref{prop:Mmwexistence} below guarantees the existence of optimal divisions, proving also that  all optimal divisions into $n=2$ subsets are balanced,  %(Proposition~\ref{prop:Mmwbalanced1})
%while for $n\geq 3$ this question remains open.
We will also obtain sharp lower and upper bounds for the corresponding optimal value in Theorem~\ref{prop:Mmwbound1}.

\begin{lemma}
\label{le:parallel} %(Strict monotonicity for parallel partitions, for $n=2$)
Let $C$ be a convex body in $\rr^d$, and let $P$ be a 2-division of $C$ into subsets $C_1$, $C_2$, given by a hyperplane $H$, with $w(C_1)<w(C)$. For $t>0$, let $P^t$ be the 2-division of $C$ into subsets $C_1^t$, $C_2^t$ given by a hyperplane $H^t$ parallel to $H$ at distance $t>0$, such that $C_1\subset C_1^t$. Then, $w(C_1)<w(C_1^t)$.
\end{lemma}
%Let $H^t$ be a hyperplane parallel to $H$ at distance $t>0$, and let $P^t$ be the division of $C$ given by $H^t$ with subsets $C_1^t$, $C_2^t$, such that $C_1\subset C_1^t$. Then $w(C_1)<w(C_1^t)$.

\begin{proof}
If $w(C_1^t)=w(C)$, the statement trivially holds. So we can assume that $w(C_1^t)<w(C)$. Let  $B$ be a slab determining the width of $C_1^t$. Then $w(B)=w(C_1^t)<w(C)$, and so there necessarily exists a point $q_2\in C_2^t\subset C$ such that $q_2\notin B$. Call $h_1$ the hyperplane from $\partial B$ which is closer to $q_2$, and let $s_t=H^t\cap C$.
We claim that $h_1\cap s_t\neq\emptyset$: otherwise, $s_t$ would be completely contained in the interior of $B$. Extend $s_t$ until meeting $h_1$. By convexity, and using the point $q_2$ and the intersection point of $h_1$ and $\partial C_1^t$ (which does not belong to $s_t$), we could find a point of $C$ in the extension of $s_t$, which is a contradiction. Therefore, we have that $h_1$ intersects $s_t$.  %(this implies that the optimal band $B$ for $C_1^t$ touches necessarily $s_t$).
If $\ptl C_1\cap h_1\neq\emptyset$, then the segment joining that intersection point with $q_2$ yields a  contradiction (since we could find again a point of $C$ in the extension of $s_t$, by convexity). Thus $\ptl C_1\cap h_1=\emptyset$, and so there exists a slab $B'$ containing $C_1$ which is strictly contained in $B$. Then,  $w(C_1)\leq w(B')<w(B)=w(C_1^t)$, as stated. % (all this is proved in a different similar way in the other notes, and this proof should be refined).
\end{proof}

\begin{theorem}
\label{prop:Mmwexistence}
Let $C$ be a convex body in $\rr^d$. Then, there exists a balanced optimal $n$-division of $C$ for the Max-min problem for the width. Moreover, any optimal $2$-division is balanced.
\end{theorem}

\begin{proof}
Let us first prove the existence of an optimal $n$-division. Let $\{P_k\}$ be a sequence of $n$-divisions of $C$  such that
$\displaystyle{\lim_{k\to\infty} \widetilde{w}(P_k)=\widetilde{w}_n(C)}$.
Denote by  $C_1^k,\ldots,C_n^k$  the subsets of $C$ provided by $P_k$,  $k\in\mathbb{N}$.
By applying Blaschke selection theorem~\cite[Th.~1.8.7]{schneider}, we can assume that, for each $i=1,\ldots,n$,
the sequence $\{C_i^k\}$ converges to a subset $C_i^\infty$ with non-empty interior: in the case that  $C_i^\infty$ has empty interior, then  $\displaystyle{0=w(C_i^\infty)=\lim_{k\to\infty} w(C_i^k)}$,
which implies that $\widetilde{w}_n(C)=0$, contradicting Lemma~\ref{le:cotatrivial}.
Thus, the subsets $C_1^\infty,\ldots,C_n^\infty$ yield a new $n$-division  $P^\infty$ of $C$ with $\displaystyle{\widetilde{w}(P^\infty)=\lim_{k\to\infty} \widetilde{w}(P_k)=\widetilde{w}_n(C)}$, which implies that $P^\infty$ is optimal.

Let us prove now that for $n=2$ any optimal partition is balanced. Let $C_1$, $C_2$ be the subsets provided by an optimal division $P$ and assume that $P$ is not balanced, say $w(C_1)<w(C_2)$. By  Lemma~\ref{le:parallel}, we can find a 2-division $P^t$ of $C$ with subsets $C_1^t$, $C_2^t$ such that $w(C_1)<w(C_1^t)\leq w(C_2^t)$.
Then $\widetilde{w}(P^t)=w(C_1^t)>w(C_1)=\widetilde{w}(P)$, which contradicts the optimality of $P$. Therefore, $P$ must be balanced.

To finish the proof, we will show now the existence of a balanced optimal $n$-division by induction on the considered number of subsets $n\geq 2$.
If $n=2$,  it has been already shown that any optimal $2$-division is balanced. Fix now $n>2$, and assume that for any convex body in $\rr^d$, there exists a balanced optimal $m$-division, for  $m<n$. Let $P$ be an optimal $n$-division of $C$. Let $H$ be one of the hyperplane cuts  dividing $C$ into two convex regions $E_1$, $E_2$ (by virtue of Remark~\ref{re:dos}), and let $P_i$ be the $n_i$-division of $E_i$ induced by $P$, $i=1,2$, with $n=n_1+n_2$. Taking into account the induction hypothesis, there exists a balanced optimal $n_i$-division $P_i'$ of $E_i$, $i=1,2$. Observe that $\widetilde{w}(P_i)\leq \widetilde{w}(P_i')=\widetilde{w}_{n_i}(E_i)$, $i=1,2$, and so
\begin{equation}
\label{eq:chain}
\widetilde{w}_n(C)=\widetilde{w}(P)=\min\{\widetilde{w}(P_1),\widetilde{w}(P_2)\}\leq\min\{\widetilde{w}(P_1'),\widetilde{w}(P_2')\}=\widetilde{w}(P'),
\end{equation}
where $P'$ is the $n$-division of $C$ determined by $P_1'$, $P_2'$. Observe that \eqref{eq:chain} implies that $P'$ is also an optimal $n$-division of $C$. If $\widetilde{w}(P_1')=\widetilde{w}(P_2')$, then $P'$ is balanced by construction, and the statement holds.
%it follows that
%\begin{equation}
%\label{eq:chain}
%w_n(C)=w(P)=\min\{w(P_1),w(P_2)\}\leq\min\{w(P_1'),w(P_2')\}=w(P'),
%\end{equation}
%which implies that $P'$ is also optimal, and balanced by construction.
On the other hand, if $\widetilde{w}(P_1')>\widetilde{w}(P_2')$, let $H^t$ be a hyperplane parallel to $H$ at distance $t\geq 0$ with respect to $E_2$, and let $E_1^t$, $E_2^t$ be the two convex regions into which $H^t$ divides $C$. Since $$\widetilde{w}_{n_1}(E_1^0)=\widetilde{w}_{n_1}(E_1)=\widetilde{w}(P_1')>\widetilde{w}(P_2')=\widetilde{w}_{n_2}(E_2)=\widetilde{w}_{n_2}(E_2^0)$$
and
$$
\widetilde{w}_{n_1}(E_1^{t_1})=0<\widetilde{w}_{n_2}(C)=\widetilde{w}_{n_2}(E_2^{t_1}),
$$
for certain $t_1>0$ large enough, it follows by continuity that there exists $t_0>0$ such that $\widetilde{w}_{n_1}(E_1^{t_0})=\widetilde{w}_{n_2}(E_2^{t_0})$. By considering a balanced optimal $n_i$-division $P_i^{t_0}$ of $E_i^{t_0}$, which exists by the induction hypothesis and satisfies $\widetilde{w}(P_i^{t_0})=\widetilde{w}_{n_i}(E_i^{t_0})$, $i=1,2$, it follows that the $n$-division $P^{t_0}$ of $C$, determined by $P_1^{t_0}$, $P_2^{t_0}$, is balanced by construction, and also optimal: since $E_2\subset E_2^{t_0}$, we have
\begin{align*}
%w_n(C)=
\widetilde{w}(P')=&\min\{\widetilde{w}(P_1'),\widetilde{w}(P_2')\}=\widetilde{w}(P_2')= \widetilde{w}_{n_2}(E_2) \leq \widetilde{w}_{n_2} (E_2^{t_0})=
\\
=&\, \widetilde{w}(P_2^{t_0})
=\min\{\widetilde{w}(P_1^{t_0}),\widetilde{w}(P_2^{t_0})\}
=\widetilde{w}(P^{t_0}),
\end{align*}
and so equality above must hold to avoid a contradiction with the optimality of $P'$.

\end{proof}

%Este remark es un ejemplo

\begin{example}
\label{re:Mmwisosceles}
\textup{The optimal division of a given convex body for this problem is not unique in general, as can be seen with the following  example.  Consider an isosceles triangle $T$ of sides $l_1$, $l_2$, $l_3$ (being $l_1$ the shortest one), relatively close to be equilateral (for instance, let the side lengths be  $4, 5, 5$).
% si el triangulo es muy alargado, entonces la anchura del triángulito de abajo no será la distancia entre $p$ y $l_1$.
Call $P_T$ to the 2-division of $T$ into subsets $T_1$, $T_2$ determined by the bisection of one of the largest angles in $T$, namely, at vertex $v_2$, see Figure~\ref{fig:isosceles}.
\begin{figure}[ht]
    \centering
    \includegraphics[width=0.8\textwidth]{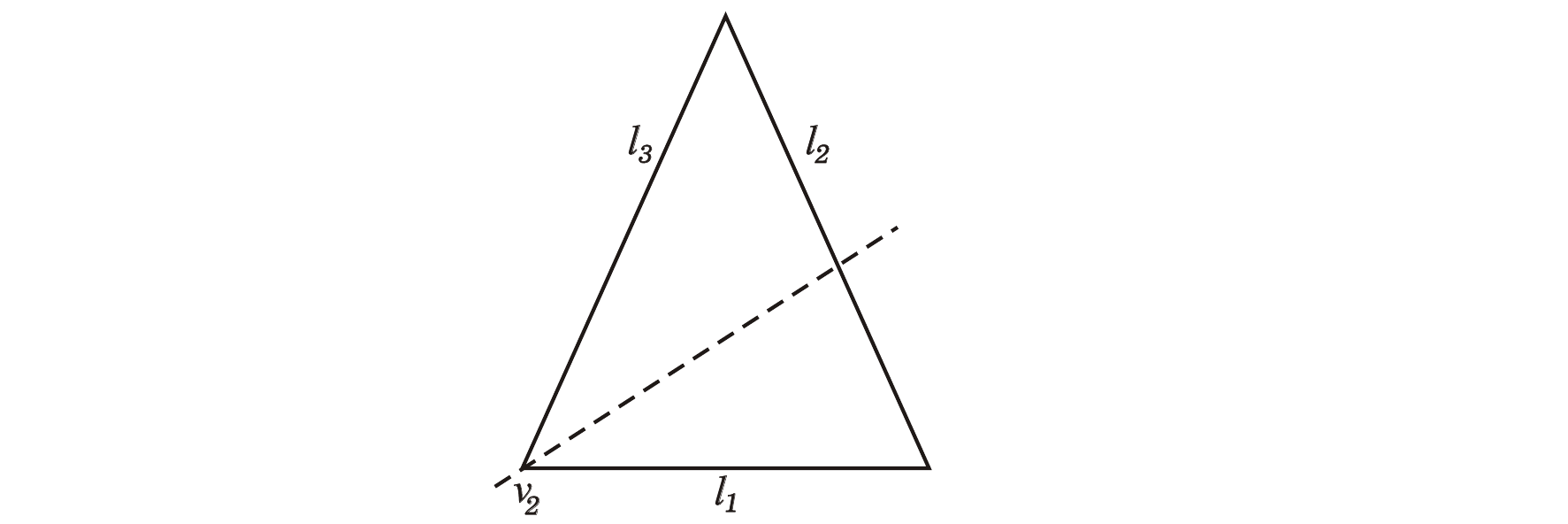}
    \caption{For an isosceles triangle $T$, the 2-division  $P_T$ is optimal}
    \label{fig:isosceles}
\end{figure}
We have that $P_T$ is a balanced optimal 2-division of $T$: it is easy to check that $w(T_1)=w(T_2)$,
% que coincidan es claro por construcción, una vez que tenemos asegurado que la anchura del triangulito de abajo es $d(p,l_1)$
and that for any other 2-division $P$ of $T$ into subsets $C_1$, $C_2$, we always have that one subset $C_i$ is  contained in the slab  determining the width of $T_j$, for some $j\in\{1,2\}$.
% con una casuística, estudiando cada partición de $T$ a partir de los puntos de corte del hiperplano y $T$.
This implies that
$$\widetilde{w}(P)=\min\{w(C_1), w(C_2)\}\leq w(C_i)\leq w(T_j)=\widetilde{w}(P_T),$$
which yields the optimality of $P_T$.
Now, call $p$ to the intersection point of the considered bisector of the angle at $v_2$  and the opposite side $l_2$, and let $q$ be the point from $l_3$ at the same height of $p$, see Figure~\ref{fig:isosceles2}.
\begin{figure}[ht]
    \centering
    \includegraphics[width=0.38\textwidth]{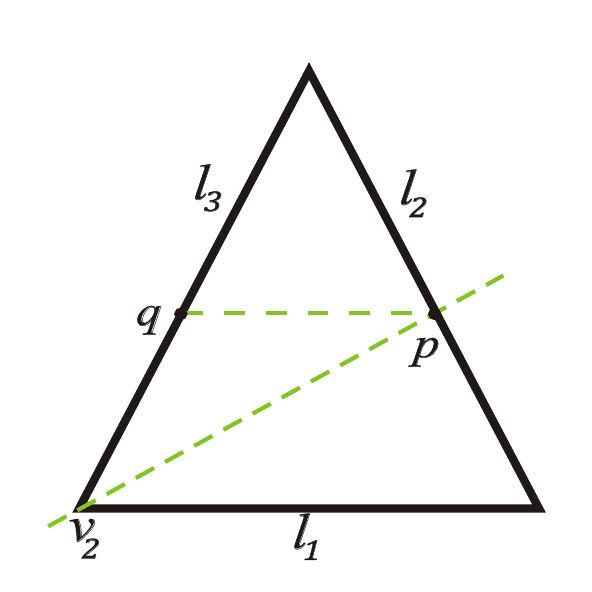}
    \caption{Any 2-division of $T$ given by a segment joining $p$ and a point of $\overline{q\,v_2}$ is optimal}
    \label{fig:isosceles2}
\end{figure}
Then, any 2-division $Q$ of $T$ given by a segment joining $p$ with a point of the segment $\overline{q\,v_2}$ is also optimal, since $w(Q)$ will coincide with $w(P_T)$ by construction.}
\end{example}

We will now focus on obtaining lower and upper bounds for the optimal value for this problem. %We firstly recall the following well-known result due to T.~Bang~\cite{bang}.
%Theorem~\ref{prop:Mmwbound1} provides a lower bound for the optimal value in this setting, by using directly Lemma~\ref{le:bang}.

\begin{theorem}
\label{prop:Mmwbound1}
Let $C$ be a convex body in $\rr^d$. %Let $w_n(C)$ be the optimal value of $C$ for the Max-min problem for the width when considering $n$-divisions.
Then,
\begin{equation}\label{eq:Mmwbound1}
\frac{w(C)}{n}\leq \widetilde{w}_n(C) \leq \min\bigg\{w(C),\frac{D(C)}{2}\bigg\}.
\end{equation}
\end{theorem}

\begin{proof}
Let $P$ be a balanced optimal $n$-division of $C$ into subsets $C_1,\ldots, C_n$, by using  Theorem~\ref{prop:Mmwexistence}.
%, we can consider a balanced optimal $n$-division of $C$ into subsets $C_1,\ldots,C_n$.
By considering the slab $B_i$ which determines  $w(C_i)$,  $i=1,\ldots,n$, and applying Lemma~\ref{le:bang}, it follows that $$w(C)\leq \sum_{i=1}^n w(B_i)=\sum_{i=1}^n w(C_i)=n\, \widetilde{w}_n(C),$$
which gives left-hand inequality in \eqref{eq:Mmwbound1}.
On the other hand, %let $P$ be an optimal $n$-division of $C$ into subsets $C_1,\ldots, C_n$ (see Theorem~\ref{prop:Mmwexistence}), and
let $H$ be one of the hyperplane cuts from $P$ dividing $C$ into two convex regions (see Remark~\ref{re:dos}).  Consider two hyperplanes $H_1$, $H_2$ parallel to $H$ and tangent to $\partial C$. Let $p_i$ be a point from $\ptl C\cap H_i$, and let $B_i$ be the slab delimited by $H$ and $H_i$, $i=1,2$. Observe that any subset $C_i$
%provided by $P$
will be contained in either $B_1$ or in $B_2$, so we can assume,   without loss of generality, that $C_i\subset B_i$, $i=1,2$. Then, denoting by $d$ to the Euclidean distance,
$$
D(C)\geq d(p_1,p_2)\geq w(B_1)+w(B_2)\geq w(C_1)+w(C_2)= 2\,\widetilde{w}(P)=2\,\widetilde{w}_n(C),
$$
and so $\widetilde{w}_n(C)\leq D(C)/2$, which together with $\widetilde{w}_n(C)\leq w(C)$ (see Lemma~\ref{le:cotatrivial}),  gives the right-hand side inequality in \eqref{eq:Mmwbound1}.
\end{proof}

If $n=2$, the equality in the left-hand side of \eqref{eq:Mmwbound1} is attained when $C$ is a constant-width body. Sharpness of the right-hand side can be proved in some cases, as shown in Example~\ref{re:234}.

%Notice that if $n=2$, equality is attained when $C$ is a constant-width body for the left-hand side of \eqref{eq:Mmwbound1}, and when $C$ is a $d$-ball for the left-hand side. The following example shows the sharpeness of the bound given in the right-hand side of the inequality in the planar case.

\begin{example}
\label{re:234}
\textup{
%The following examples show that inequalities~\eqref{eq:Mmwbound1} from Theorem~\ref{prop:Mmwbound1} are sharp.
Let $C$ be a (sufficiently) long and narrow orthotope in $\rr^d$. We can construct a balanced division $P$ of $C$, by using proper parallel lines, such that the width of all the subsets given by $P$ equals $w(C)$, see Figure~\ref{fig:Mmwrectangle}.
\begin{figure}[ht]
    \centering
    \includegraphics[width=0.75\textwidth]{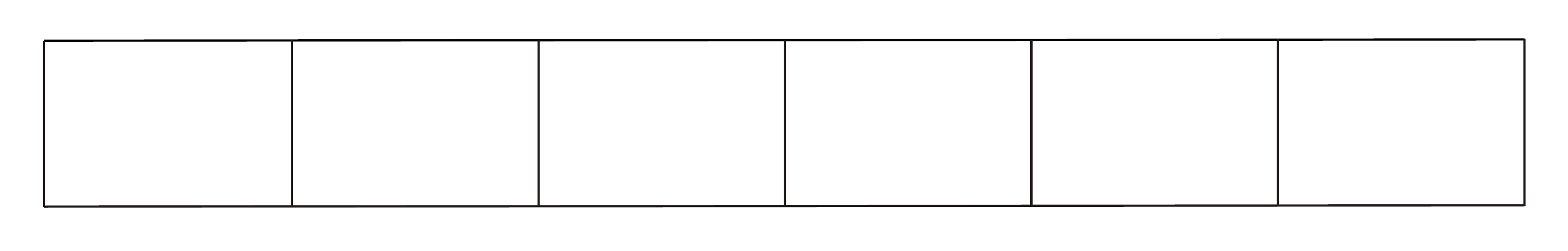}
    \caption{An optimal $6$-division of a long and narrow rectangle ($d=2$)}
    \label{fig:Mmwrectangle}
\end{figure}
Thus, Theorem~\ref{prop:Mmwbound1}  implies that $P$ is optimal and $\widetilde w_n(C)=w(C)=\min\{w(C),D(C)/2\}$. On the other hand, let $B$ be a planar ball of radius $r>0$, and let $P$ be one of the (balanced) $n$-divisions of $B$ from  Figure~\ref{fig:234}, for $n=2,\,3,\,4$. Then $w(P)=r$, which in particular gives that $P$ is optimal and $\widetilde{w}_n(B)=D(B)/2=\min\{w(B),D(B)/2\}$ (the same behaviour holds in general dimension).
%Finally, we also point out that inequality~\eqref{eq:Mmwbound1} is strict for an equilateral triangle $\mathcal{T}$, since it can be checked that $\widetilde{w}_2(\mathcal{T})=w(\mathcal{T})/2$, in view of Example~\ref{re:Mmwisosceles}.
\begin{figure}[ht]
    \centering
    \includegraphics[width=0.75\textwidth]{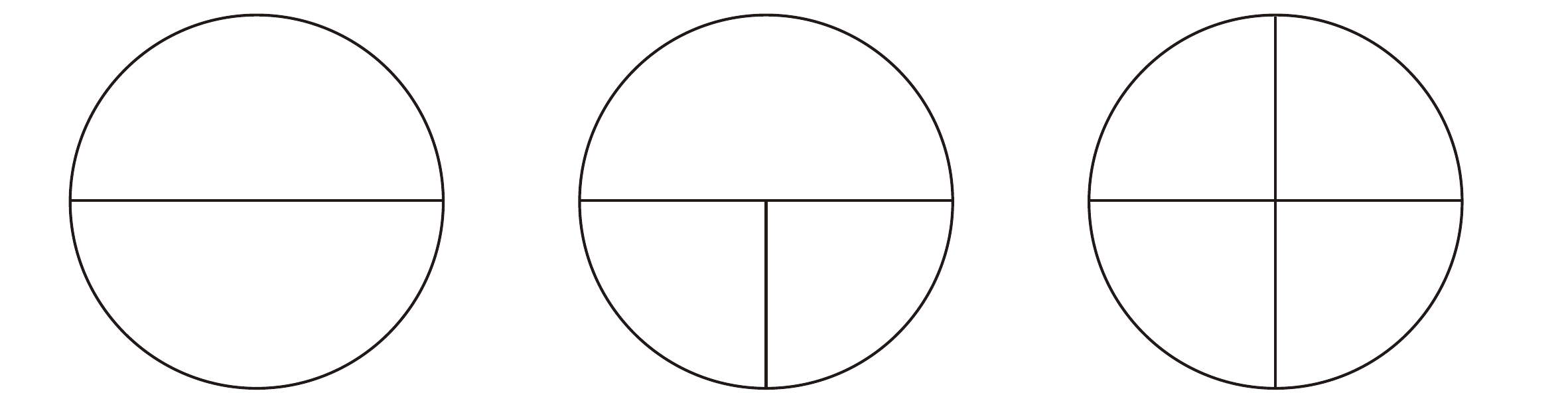}
    \caption{Optimal $n$-divisions of a planar ball for $n=2,3,4$}
    \label{fig:234}
\end{figure}
}
\end{example}

%\begin{remark}\textcolor{blue}{Tras el comentario de Alberto, he visto que la 3-division descrita no es óptima: todo se puede mejorar inclinando las dos líneas divisorias. Así que esto hay que quitarlo}
%Regarding the balancing behaviour of the optimal divisions for this problem, we believe that Proposition~\ref{prop:Mmwbalanced1} cannot be extended to divisions into an arbitrary number of subsets. In other words, we think that there exist optimal $n$-divisions for this problem which are not balanced, for $n\geq3$. For instance, consider a rectangle $C$ with sides of lengths 1 and 3. It is clear that the $3$-division $P$ of $C$, determined by two equally-spaced lines  parallel to the shortest side of $C$, is balanced and optimal, in view of Theorem~\ref{prop:Mmwbound2} (see also Remark~\ref{re:234}). Consider now a  rectangle $C'$ with sides of lengths $1+\varepsilon$ and 3, for small $\varepsilon>0$. It is reasonable thinking that the analogous balanced $3$-division $P'$ of $C'$ will be also optimal. If this is the case, we can then obtain \emph{a non-balanced optimal 3-division} of $C'$ by modifying slightly one of the line cuts of $P'$, see Figure~\ref{fig:epsilon}. Although this example is not complete, it strongly suggests that there may be non-balanced optimal $n$-divisions for this problem when  $n\geq3$.

%\begin{figure}[ht]
 %   \centering
 %   \includegraphics[width=0.85\textwidth]{epsilon}
  %  \caption{We believe that these two $3$-divisions of the rectangle $C'$ are optimal, but the right-hand side one is not balanced}
 %   \label{fig:epsilon}
%\end{figure}
%\end{remark}

\subsection{Max-min problem for the inradius}
\label{subsec:Mmi}
The Max-min problem for the inradius shares several features with the Max-min problem for the width from Subsection~\ref{subsec:Mmw}: Theorem~\ref{prop:Mmiexistence} follows by using analogous techniques as in the width case.
%However, proving that any optimal division in this setting is balanced is still an open question.
Moreover,  the optimal value for this problem when considering divisions of a convex body $C$ into $n=2$ subsets can be expressed in terms of the width of a certain {\em rounded} subset of $C$ (see Theorem~\ref{prop:Mmiov}). We point out that several issues remain open for this problem, such as refining the bounds for the optimal value or deciding whether any optimal division is  balanced.

%XXX La proof de este teorema es igual que en el caso de la anchura. No sé si merece la pena unificar las dos.

\begin{theorem}
\label{prop:Mmiexistence}
Let $C$ be a convex body in $\rr^d$. Then, there exists a balanced optimal $n$-division of $C$ for the Max-min problem for the inradius. Moreover,
\begin{equation}
\label{eq:Mmibound}
\widetilde{I}_n(C)\geq I(C)/n.
\end{equation}
\end{theorem}

\begin{proof}
For the existence of optimal divisions, it can be checked that the proof of Theorem~\ref{prop:Mmwexistence} still holds if we consider the inradius instead of the minimal width.
Let us see that we can always find a balanced optimal $n$-division by induction on the number $n$ of subsets.

For $n=2$, let $P$ be an optimal 2-division of $C$ into  subsets $C_1$, $C_2$,
which will be determined by a hyperplane $H$.
Assume that $P$ is not balanced, say $I(C_1)<I(C_2)$.
For each $t\geq0$, consider a hyperplane $H^t$ paralell to $H$ at distance $t$ from $C_1$, and let $P^t$ be the 2-division of $C$ into subsets $C_1^t$, $C_2^t$ determined by $H^t$. Since $I(C_1^0)=I(C_1)<I(C_2)=I(C_2^0)$, and $I(C_1^{t_1})=I(C)>0=I(C_2^{t_1})$, for certain large enough $t_1>0$,
%such that $C_1^{t_1}=C$, and $C_2^{t_1}$ has empty interior),
it follows by continuity that there exists $t_0\in(0,t_1)$ such that $I(C_1^{t_0})=I(C_2^{t_0})$. Thus, the 2-division  $P^{t_0}$ of $C$ is balanced and satisfies that $\widetilde{I}(P^{t_0})=I(C_1^{t_0})\geq I(C_1)=\widetilde{I}(P)$, since $C_1\subset C_1^{t_0}$. This implies that $P^{t_0}$  is also optimal, as desired.
% XX La demo de esta proposicion es parcialmente (caso $n>2$) como la de la anchura. Aquí parece más claro que no se pueden unificar.
And for an arbitrary $n>2$, we can proceed as in the proof of Theorem~\ref{prop:Mmwexistence} in order to obtain a balanced optimal $n$-division of $C$.

Finally, let now $P$ be a balanced optimal $n$-division of $C$ into subsets $C_1,\ldots,C_n$.
Lemma~\ref{le:kadets} implies that
$$
I(C)\leq\sum_{i=1}^n I(C_i)= \sum_{i=1}^n \widetilde{I}_n(C)=n\,\widetilde{I}_n(C),
$$
which yields the statement.
\end{proof}

We will now focus on estimating the optimal value of a given convex body $C$ for this problem when considering divisions into $n=2$ subsets.
The two following technical results will lead us to Theorem~\ref{prop:Mmiov}, which establishes implicitly the optimal value for this problem when $n=2$, following the same spirit as in~\cite[Th.~1]{bb} for the  min-Max problem for the inradius, see Subsection~\ref{subsec:mMi}.

\begin{lemma}
\label{le:technical}
Let $C$ be a convex body in $\mathbb{R}^d$, and let $P$ be a 2-division of $C$ into subsets $C_1$, $C_2$. Assume that $C=C^\rho$, where  $\rho=I(C_1)$. Then,  $w(C_1)=2\rho$.
\end{lemma}

\begin{proof}
Recall that it is always true that
\begin{equation}
\label{eq:wi}
2\,I(E)\leq w(E),
\end{equation}
for any convex body $E$ (note that any inball of $E$ will be contained in any slab providing $w(E)$). Let $B$ be an inball of $C_1$, with radius $\rho$, which will necessarily touch $\ptl C_1$ in at least two points.
% si toca sólo en uno, podremos agrandarla claramente
If $B$ touches $\partial C_1$ in just two points, then they must be antipodal in $B$,
%(si no, tienes espacio para agrandar un poco la inbola en $C_1$)
and the tangent hyperplanes to $B$ at those points will determine a slab containing $C_1$, due to the convexity of $C_1$.
% si hay un punto $q$ de $C_1$ no contenido en tal slab, entonces el segmento que une q y con un punto próximo al punto de contacto de $B$ y $C_1$, contenido en $C_1$ por convexidad, me dará un punto de $C_1$ "por encima del punto de contacto", lo que es contradictorio.
This gives $w(C_1)\leq 2\rho$, and so $w(C_1)=2\rho$, by using~\eqref{eq:wi}.  Otherwise, if there are (at least) three points in $B\cap \partial C_1$, it follows that $C_1$ will be contained in the region $R$ determined by the corresponding tangent hyperplanes to $B$, using again the convexity of $C_1$.
% mismo razonamiento que antes: en caso contrario, hallaremos puntos de $C_1$ "por encima" de los puntos de contacto con $B$ .
Call $H_i$ to these hyperplanes, $i=1,2,3$. Without loss of generality, we can assume that the  hyperplanes are not parallel by pairs (if  $H_i$ and $H_j$ were parallel, then two points from $B\cap\ptl C_1$ would be antipodal and we could proceed as previously). It is clear that at least two of the hyperplanes, say $H_1$, $H_2$, will not coincide with the hyperplane cut providing the division $P$. Then, since $C=C^\rho$, it follows that $C_1$ will be contained in the slab determined by $H_3$ and its parallel hyperplane at distance $2\rho$, and so $w(C_1)\leq 2\rho$, which yields again $w(C_1)=2\rho$ by using~\eqref{eq:wi}.
\end{proof}

\begin{lemma}
\label{co:technical}
Let $C$ be a convex body in $\rr^d$. Assume that  $C=C^\rho$, where $\rho=I_2(C)$.  %is the optimal value of $C$ for the Max-min problem for the inradius when considering 2-divisions.
Then,  $w_2(C)=2\rho$. %, where $w_2(C)$ is the optimal value of $C$ for the Max-min problem for the width when considering 2-divisions.
\end{lemma}

\begin{proof}
Let $P$ be a balanced optimal 2-division of $C$ into subsets $C_1$, $C_2$, in view of Theorem~\ref{prop:Mmwexistence}. By the definition of $\rho$, we will have that $I(C_i)\leq\rho$, for some $i\in\{1,2\}$. This implies that $C^\rho\subseteq C^{I(C_i)}$, due to the set inclusion property for rounded bodies.
% if $a<b$, then $C^b\subset C^a$; con bolitas pequeñas recorreré más espacio
As $C=C^\rho$ by hypothesis, it follows that $C^{I(C_i)}=C$.
Then, by using Lemma~\ref{le:technical}, we conclude that $\widetilde{w}_2(C)=\widetilde{w}(P)=w(C_i)=2\,I(C_i)\leq 2\rho$. On the other hand, the general inequality~\eqref{eq:wi} easily leads to $2\rho\leq\widetilde{w}_2(C)$ in our case, which completes the proof.
\end{proof}

\begin{theorem}
\label{prop:Mmiov}
Let $C$ be a convex body in $\rr^d$. %Let $I_2(C)$ be the optimal value of $C$ for the Max-min problem for the inradius when considering 2-divisions.
Then, $\widetilde{I}_2(C)$ is the unique number $\widetilde{\rho}$ such that
\begin{equation}
\label{eq:Mmi2}
2\widetilde{\rho}=\widetilde{w}_2(C^{\widetilde{\rho}}).
\end{equation}
%where $w_2(C^{\widetilde{\rho}})$ is the optimal value of the $\widetilde{\rho}$-rounded body $C^{\widetilde{\rho}}$ of $C$ for the Max-min problem for the width when considering 2-divisions.
\end{theorem}

\begin{proof}
First of all, note that $\widetilde{w}_2(C^\rho)$ decreases when $\rho$ increases from zero to $I(C)$,
due to the set inclusion property for rounded bodies, and moreover,
for $\rho_1=0$ we will have $2\rho_1=0<\widetilde{w}_2(C)=\widetilde{w}_2(C^{\rho_1})$,
% For $\rho=0$, we have that $2\rho=0<w_2(C)=w_2(C^0)$
while for $\rho_2=I(C)$ it follows that $2\rho_2=w(C^{\rho_2})\geq \widetilde{w}_2(C^{\rho_2})$.
This implies, by continuity,  the existence of a number $\widetilde{\rho}\in(0,I(C)]$ satisfying~\eqref{eq:Mmi2}.
Furthermore, if $\rho_1$, $\rho_2$ are two positive numbers  satisfying~\eqref{eq:Mmi2},
with $\rho_1<\rho_2$, then $\widetilde{w}_2(C^{\rho_2})\leq \widetilde{w}_2(C^{\rho_1})$ from the previous argument, but we also have that
$$
\widetilde{w}_2(C^{\rho_2})=2\rho_2>2\rho_1=\widetilde{w}_2(C^{\rho_1}),
$$
which is contradictory. Therefore, the uniqueness from the statement holds.

We will now check that $\widetilde{I}_2(C)$ satisfies~\eqref{eq:Mmi2}. Call $\tau:=\widetilde{I}_2(C)$ for simplicity.
On the one hand, observe that any balanced optimal 2-division $P$ of $C$ for the inradius
induces a 2-division $P'$ of $C^\tau$ with $\widetilde{I}(P')=\widetilde{I}(P)=\tau$, % como redondeamos con el valor I(P)=I(C_1)=I(C_2), cuando consideremos $C_i'$ (cuyo inradio será menor o igual que $C_i$ por inclusión de conjuntos), podremos meterle una bola de radio I(P).
and so $2\tau\leq 2 \widetilde{I}_2(C^\tau)\leq \widetilde{w}_2(C^\tau)$,
taking into account the definition of $\widetilde{I}_2(C^\tau)$ and~\eqref{eq:wi}.
On the other hand, $\widetilde{I}_2(C^\tau)\leq \tau$ since $C^\tau\subseteq C$,
and so $C^\tau$ coincides with its associated  $\widetilde{I}_2(C^\tau)$-rounded body.
Then, by applying Lemma~\ref{co:technical} to $C^\tau$,
we have  that $\widetilde{w}_2(C^\tau)=2 \widetilde{I}_2(C^\tau)\leq 2\tau$.
Both inequalities imply  that $\tau$ satisfies~\eqref{eq:Mmi2}, completing the proof.
\end{proof}

It seems difficult to obtain a similar result to Theorem~\ref{prop:Mmiov} for divisions into more than two subsets,
since Lemma~\ref{le:technical} does not hold in general,
as the following example illustrates.

% XXX Meter este remark como example?

\begin{example}
Let $\mathcal{T}$ be an equilateral triangle, and let $E$ be the trapezoid constructed with three copies of  $\mathcal{T}$ by joining two pairs of sides. Consider now the $\rho$-rounded body $E^\rho$ associated to $E$, for $\rho=I(\mathcal{T})$, and let $P$ be the 3-division of $E^\rho$ determined by the sides of the underlying equilateral triangles, see Figure~\ref{fig:rounded}. Call $C_1, C_2, C_3$ to the subsets of $E^\rho$ provided by $P$, which satisfy that $I(C_i)=I(\mathcal{T})$, $i=1,2,3$.  In this setting, it is clear that $E^\rho$ coincides with its associated $I(C_i)$-rounded body, and so the hypothesis from  Lemma~\ref{le:technical} holds, for $i=1,2,3$. However, it can be checked that there is one subset whose width is not equal to $2\rho$ (in fact, $P$ is not balanced for the width in this case).
\end{example}

\begin{figure}[ht]
    \centering
    \includegraphics[width=0.65\textwidth]{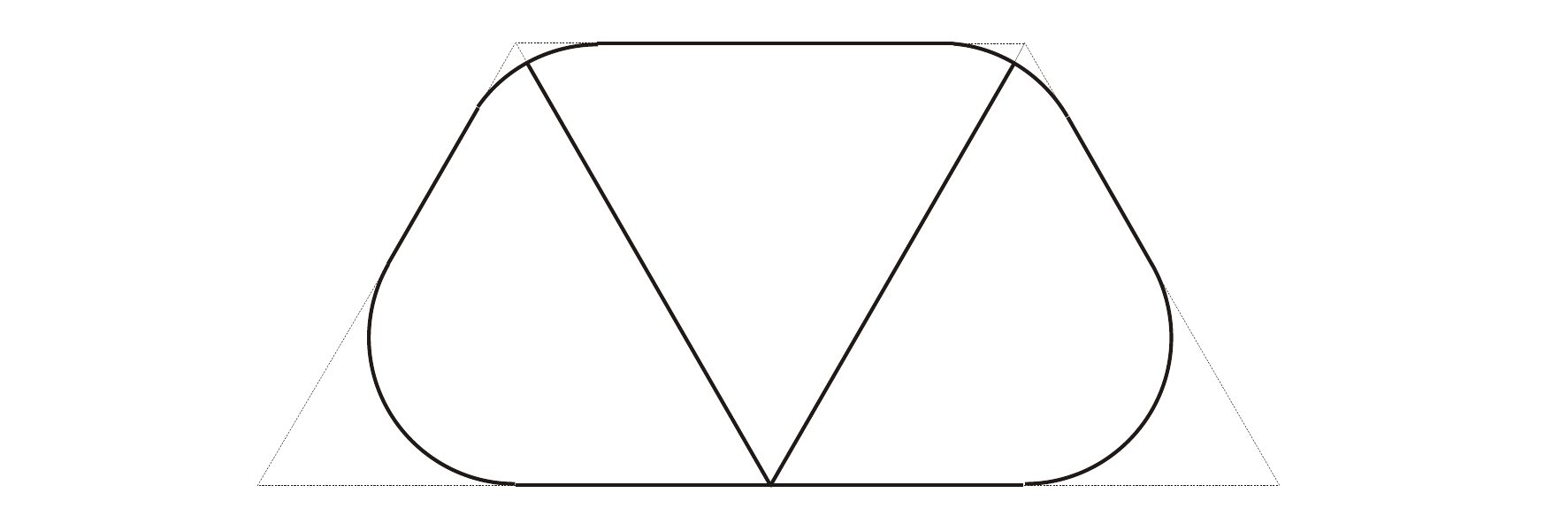}
    \caption{Lemma~\ref{le:technical} does not hold for this 3-division  of $E^\rho$}
    \label{fig:rounded}
\end{figure}

\section{Related problems}
\label{sec:final}

%\subsection{Some related open questions}
Apart from the questions which have not been completely solved in the previous Sections, other magnitudes $F$ can be considered in  the min-Max and Max-min problems, as the \emph{circumradius} (which represents the smallest radius of a ball containing the original convex body) and the {\em perimeter}. In this first case, it can be proved for instance that not all optimal divisions for the Max-min problem are balanced. The second case, with the additional restriction that the subsets of the  divisions enclose a prescribed quantity of volume, is related to the isoperimetric tilings problem~\cite[Problem~C15]{cfg}.

A nice variant of the problems treated in this work is described in~\cite[Remark~3]{bb}: for an $n$-division $P$ of a convex body $C$ into subsets $C_1,\ldots,C_n$, we can consider the quantity $\displaystyle{\sum_{i=1}^n F(C_i)}$, where $F$ is one fixed geometric magnitude, and then determine the $n$-division of $C$  minimizing (or maximizing) that quantity, as well as the corresponding optimal value.   Lemma~\ref{le:bang} (by T.~Bang) and Lemma~\ref{le:kadets} (by V.~Kadets) provide lower bounds for the optimal values in the case of the minimal width and the inradius, respectively.

Finally, a possible generalization of our work can be posed by considering general divisions, not necessarily determined by hyperplane cuts. In that case, the subsets provided by those divisions may not be convex, yielding more complicated situations.

\section*{Acknowledgements}
The first author is supported by  MICINN project PID2020-118180GB-I00, and by Junta de Andaluc\'ia grant FQM-325. The second author is supported by MICINN/ FEDER project    PID2020-118137GB-I00. The third author is supported by MICINN projects  PID2019-103900GB-I00 and PID2019-104129GB-I00/AEI/10.13039/501100011033.

\bibliographystyle{plain}

\bibliography{borsuk}

\end{document}